\documentclass[letterpaper, 10pt,reqno]{amsart}
\usepackage[utf8]{inputenc}

\usepackage{amsmath}
\usepackage{amssymb}
\usepackage{amsfonts}
\usepackage{amsthm}
\usepackage{bbm}
\usepackage[mathscr]{eucal}
\usepackage{a4wide}
\usepackage{float}
\usepackage{comment}
\usepackage{enumitem}

\usepackage[normalem]{ulem}

\usepackage{orcidlink}

\usepackage{mathtools}
\mathtoolsset{showonlyrefs}

\usepackage{hyperref}
\hypersetup{
	colorlinks = true,
	urlcolor = {magenta},
	citecolor = {red},
	linkcolor = {blue}
}
\usepackage{doi}
\usepackage[square,numbers]{natbib}

\usepackage{xcolor}
\usepackage{todonotes}

\newcommand{\RNum}[1]{\uppercase\expandafter{\romannumeral #1\relax}}

\newcommand{\indicwo}[1]{\1_{#1}}

\newcommand{\pr}{\mathbb{P}}								

\newcommand{\Prob}[1]{\pr\left(#1\right)}					

\newcommand{\e}{\mathbb{E}}								

\newcommand{\1}{\mathbbm{1}}								
\newcommand{\ind}[1]{\1_{\{#1\}}}	

\allowdisplaybreaks

\newtheorem{theorem}{Theorem}[section]
\newtheorem{lemma}[theorem]{Lemma}

\newtheorem{proposition}[theorem]{Proposition}

\newtheorem{conjecture}[theorem]{Conjecture}

\newtheorem{example}{Example}[section]

\newtheorem{problem}[theorem]{Problem}

\theoremstyle{definition}
\newtheorem{definition}[theorem]{Definition}
\newtheorem{assumption}[theorem]{Assumption}
\newtheorem{remark}[theorem]{Remark}

\numberwithin{equation}{section}

\newcommand{\cA}{\mathcal A}
\newcommand{\cC}{\mathcal C}
\newcommand{\cD}{\mathcal D}
\newcommand{\cE}{\mathcal E}
\newcommand{\cG}{\mathcal G}

\newcommand{\cN}{\mathcal N}
\newcommand{\cP}{\mathcal P}
\newcommand{\cR}{\mathcal R}
\newcommand{\cS}{\mathcal S}
\newcommand{\cU}{\mathcal U}
\newcommand{\cX}{\mathcal X}

\newcommand{\T}{\mathbb T}
\newcommand{\Z}{\mathbb Z}
\newcommand{\N}{\mathbb N}
\newcommand{\R}{\mathbb R}
\newcommand{\G}{\mathbb G}
\newcommand*{\be}{\begin{equation}}
	\newcommand*{\ee}{\end{equation}}
\newcommand*{\ba}{\begin{aligned}}
	\newcommand*{\ea}{\end{aligned}}
\renewcommand{\P}[1]{\mathbb P\left(#1\right)}
\newcommand{\wt}{\widetilde}
\newcommand{\toindis}{\overset d\longrightarrow}
\newcommand{\toinp}{\overset{\mathbb P}{\longrightarrow}}
\newcommand{\cO}{\mathcal O}
\newcommand{\eps}{\epsilon}
\renewcommand{\b}{\backslash}
\renewcommand{\e}{\mathrm e}
\newcommand{\E}[1]{\mathbb E\left[#1\right]}
\newcommand{\cF}{\mathcal F}

\newcommand{\BP}{\mathrm{BP}}
\newcommand{\dd}{\mathrm d}

\newcommand{\floor}[1]{\lfloor #1\rfloor}

\newcommand{\toas}{\overset{\mathrm{a.s.}}{\longrightarrow}}

\newcommand{\invisible}[1]{}

\title{Long-range competition on the torus}

\author{Bas Lodewijks\orcidlink{0000-0001-5624-2410}}
\address[Lodewijks]{School of Mathematical and Physical Sciences, University of Sheffield}
\email{bas.lodewijks@sheffield.ac.uk}

\author{Neeladri Maitra}
\address[Maitra]{Department of Mathematics, University of Illinois Urbana--Champaign}
\email{nmaitra@illinois.edu}

\date{}

\begin{document}

	\maketitle

    \begin{abstract}
		We study competition between two growth models with long-range correlations on the torus $\T_n^d$ of size $n$ in dimension $d$. We append the edge set of the torus $\T_n^d$ by including all non-nearest-neighbour edges, and from two source vertices $v^\ominus$ and $v^\oplus$ in $\T_n^d$ two infection processes $\ominus$ and $\oplus$ start spreading to other vertices. Each susceptible vertex can be infected by at most one infection type and when infected stays infected forever (i.e.\ competing SI models). A  vertex $v$ infected with type $\square\in\{\ominus,\oplus\}$ infects a susceptible vertex $u$ at rate $\lambda_\square \|u-v\|^{-\alpha_\square}$, where $\lambda_{\ominus}=\lambda_\ominus(n),\lambda_\oplus=\lambda_\oplus(n)>0$ and $\alpha_\ominus=\alpha_\ominus(n),\alpha_\oplus=\alpha_\oplus(n)\in[0,d)$ are allowed to depend on $n$. We study \emph{coexistence}, the event that both infections reach an asymptotically positive proportion of the graph as $n$ tends to infinity, and identify precisely when coexistence occurs. In the case of absence of coexistence, we outline  several phase transitions in the size of the infection that reaches a negligible proportion of the vertices, which depends on the ratio of the sum of infection rates across all vertices of type $\ominus$ and $\oplus$. The work extends known results for the case $\alpha_\ominus(n)=\alpha_\oplus(n)\equiv 0$ and $\lambda_\ominus(n)\equiv 1, \lambda_\oplus(n)\equiv \lambda>0$, and includes general and novel results that cannot be observed when the model parameters are fixed and independent of $n$. The main technical contribution is a coupling of the competition process with branching random walks, where we are able to use the coupling even when  coupling error between the competition process and the branching random walks is of the same order of magnitude as the size of the coupled processes. 
\end{abstract}

        
	\section{Introduction}
Suppose that a virus spreads through a population, when suddenly a mutated variant appears. This mutated variant is less contagious, but has a longer incubation period (the time until symptoms develop). As such, though the variant spreads slower than the original virus locally, its increased incubation period may result in infections to occur over greater distances. People who carry the virus but show no symptoms may still go into work, visit friends or family, or travel large distances. As such, the variant can more easily spread along greater distances compared to the original virus.

An important question is which of the two viruses is the more dangerous one. Indeed, in the outbreak of, for example, the COVID virus, it was observed that mutated variants showed behaviour different to that of the original strain~\cite{Hadj22,ZhaHuaZhaCheGaoJia22}, thus requiring a shift in focus on how to combat the outbreak. Depending on the difference in properties and behaviour of the two viruses (e.g.\ their contagiousness and incubation periods), which of the two viruses infects only a negligible part of the population, compared to the other, by the time the viruses are endemic? Or, do both viruses infect a positive proportion of the population and thus both form an equal threat? 

The aim of this paper is to study a mathematical model of competing viral infections in a population spread out through space, where the viral infections can differ not only in the \emph{local intensity} at which they spread, but also in the strength of their respective \emph{long-range} spreading rate. As far as the authors are aware, this initiates the study of such competition models where long-range effects are considered. We model the spread of infections as two long-range susceptible-infected (SI) models on the torus $\T_n^d$ of volume $n$ in dimension $d$. Equivalently, one can view this as two long-range first-passage percolation models. Let us first review some related first-passage percolation models before we describe the model studied here in more detail. 

First-passage percolation (FPP) was first introduced by Hammersley and Welsh in 1965 \cite{Hammersley1965}. It has been studied extensively as a prototypical model of fluid-flow through porous media; we refer the interested reader to the survey \cite{auffinger201750} and references therein. Informally, one considers a graph $G=(V,E)$, and associates i.i.d.\ copies $\gamma_e$ of some random variable $\gamma$ to every edge $e \in E$. We interpret $\gamma_e$ as the time it takes for a fluid to traverse the edge $e$. For the special case when $\gamma$ is an exponentially distributed random variable with mean $1/\lambda$, we say that the FPP process `flows at rate $\lambda$'. 

A model where two FPP processes compete to occupy sites in the $\mathbb{Z}^d$ was first introduced as the two-type Richardson model in \cite{Hag_Pem_98}. In this model, two FPP processes start flowing from two different sources at potentially different rates, and they compete to cover the entire lattice. The main interest is in understanding when both processes cover infinitely many sites with positive probability; an event known as \emph{coexistence}. In this direction, the main conjecture from \cite{Hag_Pem_98} is that coexistence occurs if and only if the rates of the two competing processes are equal. Note that by a straightforward scaling, one can without loss of generality consider the case where one first-passage percolation process flows at rate $1$, while the other flows at rate $\lambda\geq 1$. Many results related to this conjecture have been obtained. The original paper \cite{Hag_Pem_98} that introduced this model showed coexistence occurs for $\lambda=1$, \cite{Hagg_Pem_00} showed non-coexistence for almost all parameter values of $\lambda$, and \cite{ahlberg2020two} showed coexistence if and only if $\lambda=1$ on the half-plane $\mathbb Z\times \mathbb N$. We refer the interested reader to the survey \cite{mia_survey} and the references therein for more in-depth information.

A model of \emph{long-range first-passage percolation} (LRFPP) that inspired the model of the present paper was introduced by Chatterjee and Dey in~\cite{SC_PD_LRFPP}. Here, the underlying graph $\G_\infty=(V(\G_\infty),E(\G_\infty))$ has the $d$-dimensional integer lattice $\Z^d$ as its vertex set, and each vertex is connected to every other vertex; thus creating an `infinite complete graph embedded in $\Z^d$'. Every edge $e$ is assigned a transmission time  $T_e:=\|e\|^\alpha E_e$, where $\|e\|$ denotes the spatial length of the edge (one can take any $p$-norm to determine the length), $\alpha\geq 0$ is the \emph{long-range} parameter, and $(E_e)_{e\in E(\G_\infty)}$ is a family of i.i.d.\ rate-one exponential random variables. One can view this setting as a first-passage percolation model where transmission times are penalised by the lengths of the edge, where long edges receive heavier penalties and $\alpha$ controls the discrepancy between penalties of short and long edges. Note that $\alpha=\infty$ can be interpreted as the case of nearest-neighbour FPP on the integer lattice $\mathbb{Z}^d$.

In this paper, we study a competition variant of the LRFPP model on the torus $\T_n^d$ of volume $n$ in dimension $d$ obtained by including the non-nearest-neighbour edges (instead of the entire lattice $\Z^d$ as in the previous paragraph). Here, initially from two different source vertices $v^\ominus$ and $v^\oplus$ in $\T_n^d$, two different LRFPP processes start, say a type $\ominus$ and type $\oplus$ process, that infect vertices in $\T_n^d$. A vertex infected with type $\square\in\{\ominus,\oplus\}$ infects a susceptible vertex $u$ at rate $\lambda_\square \|u-v\|^{-\alpha_\square}$, where $\lambda_\ominus,\lambda_\oplus>0$ are \emph{global} intensities and $\alpha_\ominus,\alpha_\oplus\geq 0$ are the \emph{long-range} parameters. Referring back to  the real-world setting that served as motivation at the start, we can view  the different global intensities as the 
contagiousness of the viruses, and the long-range parameters as the strength of the long-range spread of the viruses (possibly determined by e.g.\ different incubation periods). Our main interest is in understanding the size of the set of vertices infected by a either infection once all the vertices of $\T_n^d$ are infected.

We restrict our attention to long-range parameters strictly less than the dimension $d$, which corresponds to the \emph{instantaneous percolation regime} for LRFPP studied by Chatterjee and Dey~\cite{SC_PD_LRFPP} and can be thought of as a setting with weak spatial dependence. We remark that, by~\cite[Theorem $1.2$]{SC_PD_LRFPP}, a LRFPP process on $\G_\infty$ with long-range parameter $\alpha<d$ has the following property: for any time $t>0$, the set of vertices that can be reached from the origin equals $\Z^d$ almost surely. Hence, any vertex is reached instantaneously. This clarifies the choice to restrict the competition process to the finite torus $\T_n^d$. Indeed, such a process on $\G_\infty$ is not well-defined in the instantaneous percolation regime, since one cannot distinguish which infection reaches a vertex first. For long-range parameters exceeding $d$, the competition process can be studied on $\G_\infty$ and is on-going future work. 

In this paper we are generally concerned with the occurrence (absence of) coexistence at a linear scale: after all vertices are infected, have both infection processes infected an asymptotically positive proportion of the torus, or does one process infect almost all vertices, leaving only $o(n)$ many for the other process? In the case of non-coexistence, i.e.\ one infection process infects almost all vertices of the torus, what can further be said about the size of the process that reached only a negligible proportion of vertices? We provide a complete picture in terms of the intensities $\lambda_\ominus,\lambda_\oplus$, and the long-range parameters $\alpha_\ominus,\alpha_\oplus$. In particular, we allow these parameters to depend on $n$, which uncovers more subtle behaviour in certain cases that cannot be observed from fixed parameters, independent of $n$, which is common in most of the literature.

\subsection{Model definition} \label{sec:model_def}
Fix $d\in\N$ and let $\T_n^d$ be the $d$-dimensional torus of size $n$, i.e.\ $\T_n^d=$ $[-n^{1/d}/2,n^{1/d}/2] \cap \Z^d$, where we identify vertices on opposing boundaries to be the same. Let $\lambda_\ominus,\lambda_\oplus=\lambda_\ominus(n),\lambda_\oplus(n)>0$ and $\alpha_\ominus,\alpha_\oplus=\alpha_\ominus(n),\alpha_\oplus (n)\in[0,d)$. Throughout the paper, when there is no confusion, we omit the argument $n$ and write $\lambda_\square$ (resp.\ $\alpha_\square$) for $\lambda_\square(n)$ (resp.\ $\alpha_\square(n)$), for $\square\in\{\ominus,\oplus\}$. Let $\|\cdot\|$ denote the torus norm, which we can allow to be any $\ell^p$ distance on $\T_n^d$, with appropriate periodicity conditions that come from identifying the boundaries of $\T_n^d$.

We define the Long-Range Competition (LRC) process as follows. Fix two distinct vertices $v^\ominus$ and $v^\oplus$ in $\T_n^d$. We let $(\cN_n^\ominus(t),\cN_n^\oplus(t))_{t\geq 0}$ be a Markov process with state space the set of all pairs of disjoint subsets of $\T_n^d$,
\begin{align*}
	\Omega=\{(A,B):A,B \subset \T_n^d, A\cap B = \emptyset\}.   
\end{align*}
At $t=0$, we initialise the process by setting $\cN_n^\ominus(0)=\{v^\ominus\}$ and $\cN_n^\oplus(0)=\{v^\oplus\}$. At any time $t\geq 0$ a vertex  $v \in \cN_n^{\square}(t)$ spreads to a vertex $w \in (\cN_n^\ominus(t)\cup\cN_n^\oplus(t))^c$ at rate $\lambda_\square \| v-w\|^{-\alpha_\square}$ for $\square \in \{\ominus,\oplus\}$. When $v$ spreads to $w$ at time $t\geq 0$, we add $w$ to $\cN_n^{\square}(t)$ and $w$ starts to spread  in an equivalent manner.

The LRC process, which can be viewed as a competing first-passage percolation model (also referred to as a susceptible-infected (SI) model) describes the evolution of two competing infections on $\T_n^d$, started from $v^\ominus$ and $v^\oplus$, where the parameters $\lambda_\square$ and $\alpha_\square$, as well as the geometry of $\T_n^d$, influence the rate at which an infected vertex spreads to a susceptible vertex. 

It is clear that the process reaches an absorbing state once $\cN_n^\ominus(t)\cup\cN_n^\oplus(t)=\T_n^d$ for some $t\geq0$, i.e.\ once the torus has been covered. As such, we define the \emph{cover time} as 
\be 
T_{\mathrm{cov}}:=\inf\{t\geq 0\colon \cN_n^\ominus(t)\cup\cN_n^\oplus(t)=\T_n^d\}.
\ee 
Further, we note that multiplying all spreading rates with a positive constant changes the process only by an inverse time shift. That is, for any $\kappa=\kappa(n)>0$, let $(\cN_n^{\ominus,\kappa}(t),\cN_n^{\oplus,\kappa}(t))_{t\geq 0}$ denote the LRC process where all spreading rates are multiplied by $\kappa$. We then have 
\be 
(\cN_n^{\ominus,\kappa}(t),\cN_n^{\oplus,\kappa}(t))\overset{\mathrm{d}}{=} (\cN_n^\ominus(t/\kappa),\cN_n^\oplus(t/\kappa))\qquad \text{for all }t\geq 0.
\ee 
We define 
\be\label{eq:Rndef}
R_{j,n}^\square:=\sum_{\substack{v\in\T_j^d\\ v\neq 0}}\lambda_\square(n)\|v\|^{-\alpha_\square(n)}\qquad \text{for }\square \in \{\ominus,\oplus\}.
\ee  Throughout the paper, we sometimes conveniently vary the \emph{volume} $j\in [n]$ of the torus in the sum above, but the dependence on $n$ of the quantities $R,\lambda$ and $\alpha$ is always implicit for us. We thus hide the dependence on $n$ for notational simplicity, and simply write $R^\square_{j}=R^\square_{j,n}, \lambda_\square=\lambda_\square(n)$ and $\alpha_\square=\alpha_\square(n)$. 

Using these conventions, we refer to $R_n^\square$ as the \emph{total rate} of the Type $\square$ infection (with $\square\in\{\ominus,\oplus\}$). In the remainder of the paper, we shall work only with the LRC process $(\cN_n^{\ominus,\kappa}(t),\cN_n^{\oplus,\kappa}(t))_{t\geq 0}$ with $\kappa=1/R_n^\ominus$ and we omit the superscript $\kappa$ for ease of writing. 

\textbf{Notation.} \  For sequences $(a_n)_{n\in\N}$ and $(b_n)_{n\in\N}$ such that $b_n>0$ for all $n$, we let $a_n=o(b_n)$ and $a_n=\cO(b_n)$ denote $\lim_{n \to \infty}\frac{|a_n|}{b_n}=0$ and $\limsup_{n \to \infty}\big|\frac{a_n}{b_n}\big|<\infty$, respectively. For a sequence $(X_n)_{n\in\N}$ of random variables and a random variable $X$, we write $X_n\toindis X, X_n\toinp X, X_n\toas X$ to denote convergence of $X_n$ to $X$ in distribution, in probability, and in the almost sure sense, respectively. For random variables $X$ and $Y$, we let $X \stackrel{d}{=} Y$ denote equality in distribution. We say that a sequence of events $\cE_n$ holds \emph{with high probability}, if $\P{\cE_n} \to 1$ as $n \to \infty$. Throughout the paper, when we write sums with non-integer summation bounds, we implicitly either arbitrarily take a floor or a ceiling for the bounds. Since our results are asymptotic, this notational abuse does not create any problem.

\subsection{Statements of main results}\label{sec:statements}

The main quantities of interest are the final sizes $|\cN_n^\ominus(T_{\mathrm{cov}})|$ and $|\cN_n^\oplus(T_{\mathrm{cov}})|$ of the two infections when the entire torus is infected. In particular, we are interested in whether both infections reach a linear portion of the vertices in $\T_n^d$, or if just one of the infections reaches all but a sub-linear portion of the vertices. To this end, we set 
\be \label{eq:Mdef}
M_n:=\min\big\{|\cN_n^\ominus(T_{\mathrm{cov}})|,|\cN_n^\oplus(T_{\mathrm{cov}})| \big\},
\ee 
and define \emph{coexistence} and \emph{absence of coexistence} as follows: 

\begin{definition}[(Absence of) coexistence]
We say that \emph{coexistence} occurs when 
	\be \label{eq:def_coex}
	\lim_{\eps \downarrow 0}\liminf_{n \to \infty}\mathbb P\bigg(\frac{M_n}{n}>\eps\bigg)=1.
	\ee
    We say that \emph{absence of coexistence} occurs when 
\begin{equation}\label{eq:def_abs_coex}
    \frac{M_n}{n}\toinp 0.
\end{equation}
\end{definition}

\noindent Our main theorems describe precisely when (absence of) coexistence occurs. We recall $R_n^\ominus$ and $R_n^\oplus$ from~\eqref{eq:Rndef} and define 
\be\label{eq:Zn}
Z_n:=\frac{R_n^\oplus}{R_n^\ominus}.
\ee 
Thus, $Z_n$ is the ratio of of the rate of the $\oplus$ infection to the $\ominus$ one. If only a single vertex in $\T_n^d$ would be infected with $\ominus$ (resp.\ $\oplus$), the vertex spreads at rate $1$ (resp.\ rate $Z_n$).  The ratio $Z_n$ is crucial in understanding and formulating the behaviour of the LRC process. For the presentation of the main results, it is convenient to further define 
\be  \label{eq:def_cn}
c_n:=(Z_n-1)\log n, \quad \text{or, equivalently,}\quad Z_n=1+\frac{c_n}{\log n}.
\ee
In the current paper, we work with the following assumption:

\begin{assumption}\label{ass:par}
	We have $\sup_{n\in\N}\alpha_\ominus(n), \sup_{n\in\N}\alpha_\oplus(n)\in[0,d)$. 
\end{assumption}
\noindent This assumption allows us to avoid certain technical difficulties that can arise when $\alpha_\square\uparrow d$ as $n\to\infty$ for $\square\in\{\ominus, \oplus\}$. Furthermore, the case $\alpha_\square(n)>d$ yields entirely different behaviour (see the results for a one-type infection by Chatterjee and Dey in~\cite{ChaDey16}), which we leave for future work. 

The behaviour of the LRC process can be characterised fully by the behaviour of $c_n$. Our first result determines when coexistence and absence of coexistence occur. 

\begin{theorem}[Coexistence vs.\ absence of coexistence]\label{thrm:coex}
Let $(\cN_n^\ominus(t),\cN_n^\oplus(t))_{t\geq 0}$ denote an LRC process on $\T_n^d$, where $\cN_n^{\square}(0)=v^{\square}$ for $\square\in\{\ominus,\oplus\}$, as defined in Section~\ref{sec:model_def}. Recall $c_n$ from \eqref{eq:def_cn}. Under Assumption \ref{ass:par}, 
\begin{itemize}
    \item Coexistence, in the sense of \eqref{eq:def_coex}, occurs when $\sup_{n\in\N}|c_n|<\infty$.
    \item Absence of coexistence, in the sense of \eqref{eq:def_abs_coex}, occurs when $\lim_{n\to\infty} |c_n|=\infty$. 
\end{itemize} 
\end{theorem}

In the case of absence of coexistence, we can further precisely determine the asymptotic size of the infection that occupies a sub-linear portion of the vertices in $\T_n^d$, based on the rate at which $c_n$ tends to infinity. 

We first recall that, for $p\in(0,1)$, a random variable $X$ is \emph{geometrically distributed with success parameter $p$}, and write $X\sim \mathrm{Geo}(p)$, when it has probability mass function
\be \label{eq:geo}
\P{X=k}=p(1-p)^{k-1}, \qquad k\in\N.
\ee 
We then have the following result. 

\begin{theorem}[Absence of coexistence]\label{thrm:nocoex}
	Let $(\cN_n^\ominus(t),\cN_n^\oplus(t))_{t\geq 0}$ denote an LRC process on $\T_n^d$, where $\cN_n^{\square}(0)=\{v^{\square}\}$ for $\square\in\{\ominus,\oplus\}$, as defined in Section~\ref{sec:model_def}. Recall $M_n$ from~\eqref{eq:Mdef} and $c_n$ from \eqref{eq:def_cn}. Under Assumption \ref{ass:par}, 
	\begin{enumerate}[label=(\roman*)]
		\item\label{case:1} If $|c_n|/(\log n)^2 \to \infty$, then $M_n=1$ with high probability.
		\item\label{case:fin} If $c_n=(c+o(1))(\log n)^2$ for some $c\neq 0$, then $M_n\toindis X$, where $X \sim \mathrm{Geo}(\e^{-1/|c|})$.
		\item\label{case:liminf} If $c_n=o((\log n)^2)$ but $\lim_{n\to\infty}|c_n|=\infty$, then $\log(M_n)\frac{Z_n}{\log n} \toinp 1$.
		\item\label{case:div} If $c_n=o(\log n)$ but $\lim_{n\to\infty}|c_n|=\infty$, then $|c_n|^{-1}(\log n-\log M_n) \toinp 1$.
	\end{enumerate}
\end{theorem}

\begin{remark}
 Note that in cases $(iii)$ and $(iv)$, $M_n$ tends to infinity in probability, in contrast to the cases $(i)$ and $(ii)$ where $M_n$ converges in distribution without rescaling.
\end{remark}

The main technique in proving Theorems~\ref{thrm:coex} and~\ref{thrm:nocoex} is a coupling of the two infection processes with (continuous-time) branching random walks. Although coupling such infection processes to branching processes (rather than branching random walks) is not novel in and of itself, the main contributions of this work lie mainly in the following: i) We show that the branching random walk coupling can be leveraged much further beyond the point when the coupling ceases to be exact, the usefulness of which was not exploited in the literature before, ii) We extend the standard branching process coupling to a branching random walk coupling in our spatial setting with more than one infection type, and iii) We allow the ratio $Z_n$ of rates (recall \eqref{eq:Zn}) to depend on $n$, which is not common in the literature. We refer the reader to Section \ref{sec:main_cont} for a more detailed discussion regarding the contribution of this work.

In the remainder of this section, we provide a brief overview of the proofs of Theorems \ref{thrm:coex} and \ref{thrm:nocoex}, after which we describe the main contributions of this work in more detail. We conclude with some examples and state a few conjectures and open problems. 

\begin{remark}\label{rem:wlog}
    In the remainder of this article we make the following assumptions without loss of generality, which help simplify calculations and notation:
    \begin{itemize}
        \item We set $\lambda_\ominus(n)=1$ for all $n\in\N$ and $\lambda_\oplus(n)=\lambda(n)$ for some $\lambda(n)>0$. 
        \item By symmetry, we can assume that $Z_n\geq 1$ for all $n\in\N$.
        \item Throughout we work with the process $(\cN_n^{\ominus,\kappa}(t),\cN_n^{\oplus, \kappa}(t))_{t\geq 0}$ with $\kappa=1/R_n^\ominus$. We suppress the $\kappa $ in the superscript for ease of writing.
    \end{itemize}
\end{remark}
Assuming that $Z_n\geq 1$ implies that, when absence of coexistence occurs, the smaller infection size $M_n$ is attained by the $\ominus$ infection.

\subsection{Overview of the proof} 

\label{sec:overview}

Let us discuss, at a heuristic level, the main ideas that we use to prove Theorems~\ref{thrm:coex} and~\ref{thrm:nocoex}. 

\subsubsection{Coexistence} Let us start with the main proof idea of the first part of Theorem~\ref{thrm:coex}, as this is more straightforward to describe. The key tool we deploy is a coupling of the LRC process with two independent branching random walks. That is, we couple the spread from the vertices $v^\ominus$ and $v^\oplus$ with two independent branching random walks  (see Section~\ref{sec:prelims} for definitions) $\cX_n^\ominus(t)$ and $\cX_n^\oplus(t)$ in which particles produce children at rate $1$ and $Z_n$, respectively, such that $\cN_n^{\square}(t)\subseteq \cX_n^\square(t)$ for $\square\in\{\ominus, \oplus\}$. 

This roughly works in the following way. When vertices are infected, other infected vertices can no longer spread to these newly infected vertices, which decreases the rate at which the infection spreads. To counter this depletion of vertices, we augment the graph with \emph{artificial} vertices to which the infections can spread, so that the rate at which either the $\ominus$ or $\oplus$ infection spreads is linear its current size. In the branching random walks $\cX^\square(t)$, the rate at which particles are born is also linear in the size of the process, we can couple the augmented infection processes to the two branching random walks.

This yields an upper bound on the sizes of the infection sets at any time $t\geq0$ by the respective sizes of the branching random walks. We express the total size of the Type $\square$ infection at time $t$ as $|\cN_n^{\square}(t)|=|\cX_n^\square(t)|-C_n^{\square}((0,t])$ for $\square\in\{\ominus,\oplus\}$, where the quantity $C^{\square}((0,t])$ is to be interpreted as the `defect' in coupling the infection sets with the branching random walks. This defect arises by allowing infected vertices to `re-infect' other, already infected, vertices. These re-infections correspond to births in the branching random walk whilst such re-infections do not actually occur in the LRC process, thus causing an error or defect in the coupling. Using the condition $\sup_n c_n< \infty$ of Theorem~\ref{thrm:coex} and using a point-process representation of the defect, we show that with high probability,
\begin{itemize}
	\item[a.] By the time $t_n\approx \log (n)/Z_n$, the size of each branching random walk is $\delta n$ for some $\delta>0$, and
	\item[b.] The defect $C^{\square}((0,t_n])$ is smaller than $\delta' n$ for some $\delta'<\delta$, for both $\square\in\{\ominus,\oplus\}$. 
\end{itemize}
As a result, the size of the infection sets $|\cN_n^\ominus(t_n)|$ and $|\cN_n^\oplus(t_n)|$ themselves are both at least $\eps n$ for some $\eps=\eps(\delta,\delta')\in(0,\delta)$. In particular, the final sizes $N_n^\ominus=|\cN_n^\ominus(T_{\mathrm{cov}})|$ and $N_n^\oplus=|\cN_n^\oplus(T_{\mathrm{cov}})|$ are both at least $\eps n$, which implies coexistence.

\subsubsection{Absence of coexistence}\label{sec:overview_absence}

Let us now discuss how to prove the second part of Theorem \ref{thrm:coex}, i.e.\ that absence of coexistence occurs  when $c_n\to\infty$. To prove this, we divide the time interval $(0,T_{\mathrm{cov}}]$, which starts from the initial infections and ends when all the vertices are infected, into two phases, namely $(0,t_n]$ and $(t_n,T_{\mathrm{cov}}]$, where $t_n\approx \log(n)/Z_n$. We use the coupling of the competition process to two branching random walks, as described above, to show that at the end of the first phase, $|\cN_n^\ominus(t_n)|\approx \exp(\log(n)/Z_n)$, while $|\cN_n^\oplus(t_n)|\geq \xi n$ for some constant $\xi \in (0,1)$. Again, this follows by using the growth rate of the respective branching random walks, combined with showing that the defect $C^{\square}((0,t_n])$ is not `too large' for both $\square\in\{+,-\}$. It is readily checked that $\exp(\log(n)/Z_n)=\exp(\log(n)/(1+c_n/\log n))=o(n)$ when $c_n \to \infty$, so that, at the end of the first phase, the $\ominus$ infection has covered a negligible portion of the vertices, while the $\oplus$ infection has covered a positive fraction of the vertices. We know that the $\oplus$ infection is significantly larger with high probability due to Remark~\ref{rem:wlog}.

For the second phase $(t_n,T_{\mathrm{cov}})$, we use a  P\'olya urn type argument to show that the significant advantage achieved by the $\oplus$ infection at the end of the first phase is strong enough to ensure that the size of the $\ominus$ infection does not increase beyond a multiplicative factor of its size at the end of the first phase.  Overall, at time $T_{\mathrm{cov}}$, the size of the $\ominus$ infection is thus still $o(n)$, which shows there is no coexistence of the two types. The P\'olya urn argument is based on a discrete-time description of the competition process, where the initial configuration of the urn (which equals the infection sizes at the end of the first phase) grows with $n$. 

\subsubsection{Size of the losing type} Finally, we describe the more detailed results pertaining the exact asymptotic size of $M_n$ in the different cases of Theorem~\ref{thrm:nocoex}. Again, by Remark~\ref{rem:wlog}, we know that $M_n=N_n^\ominus$ with high probability.

For Theorem \ref{thrm:nocoex} $(i)$, we modify a result from~\cite[Theorem 1.1]{HofLod23}, which tells us that a single $\oplus$ infection takes roughly $2\log(n)/ Z_n$ amount of time to reach all vertices. On the other hand, the time for vertex $v^\ominus$, which is initially infected by the $\ominus$ infection, to reach \emph{any} other vertex, is an exponential random variable with rate roughly $1$. Hence, when $c_n/(\log n)^2$ diverges or, equivalently, $Z_n/\log n$ diverges, it follows that the $\oplus$ infection reaches all vertices by time $2\log(n)/Z_n=o(1)$, whilst the $\ominus$ infection does not spread to any vertex in that time with high probability. As a result $N_n^\ominus=1$ with high probability.

Theorem \ref{thrm:nocoex} $(ii)$ uses almost the same approach, now modifying another result from~\cite[Theorem 1.1]{HofLod23}, namely that by time $\log(n)/Z_n$, a single $\oplus$ infection reaches \emph{almost all} vertices, i.e.\ only a negligible proportion of vertices are not reached. At the same time, since $c_n=(c+o(1))(\log n)^2$ or, equivalently, $Z_n=(c+o(1))\log n$, it follows that in time $\log(n)/Z_n=c^{-1}+o(1)$, the $\ominus$ infection can reach only a finite number of vertices (in the sense that this amount does not increase with $n$) since it initially infects other vertices at rate approximately $1$. To obtain the exact limiting distribution of the size of the $\ominus$ infection, we use the coupling with the two independent branching random walks and show that up to time $t_n:=(1-\delta_n)\log(n)/Z_n$, for some $\delta_n=o(1)$, the coupling of the $\ominus$ infection with a rate-one branching random walk $\cX_n^\ominus(t)$ is in fact \emph{exact} with high probability. That is, the defect $C^\ominus((0,t_n])$ equals zero with high probability. We then use a similar P\'olya urn argument as discussed in Section \ref{sec:overview_absence} to show that the uninfected vertices remaining at time $t_n$ (of which there are $o(n)$ many) are then \emph{all} infected by the $\oplus$ infection, since it has such an overwhelming advantage by time $t_n$.

For parts $(iii)$ and $(iv)$ of Theorem \ref{thrm:nocoex}, we leverage the proof  of absence of coexistence discussed above in Section \ref{sec:overview_absence}, which tells us that at time $t_n$ the $\ominus$ infection is of size at most $An^{1/Z_n}$ for a large constant $A=A(\eps)$ with probability at least $1-\eps$. This yields a  sufficient upper bound for $N_n^\ominus$. A matching lower bound follows by once more using the coupling to the branching random walks. Here we again show that up to time $t_n\approx \log(n)/Z_n$, the defect $C^\ominus((0,t_n])$ of the coupling of the $\ominus$ infection with a Yule process is not `too large'. Hence, we can bound the size of the $\ominus$ infection from below by the size of the branching random walk at time $t_n$ (which is roughly $n^{1/Z_n}$), minus a negligible term coming from the defect. This implies a lower bound on $N_n^\ominus$ of the right order.

\subsection{Related literature and main contribution}
As far as the authors are aware, the LRC model studied in this paper is the first of its kind, in the sense that it is the first competing first-passage percolation model where the edge-weights depend on the spatial embedding of the vertices in the graph. It is inspired by the LRFPP model of Chatterjee and Dey~\cite{SC_PD_LRFPP}, where the behaviour of a single type infection on $\Z^d$ is studied. This behaviour is not limited to $\alpha<d$, and the study of the competition version of this model for $\alpha\geq d$ is future work. The qualitative results in~\cite{SC_PD_LRFPP} for the regime $\alpha<d$, however, play no role in the analysis here. Instead, the quantitative results of the first author and Van der Hofstad~\cite{HofLod23} regarding the single type LRFPP on $\T_n^d$ are used to a greater extent. Moreover, a spatial variation on the coupling of the first-passage percolation process to a branching process as in~\cite{BhaHof12}, underpins the analysis here.

Competing first-passage percolation on finite random graphs, has received a significant interest, where many well-known models have been studied. These include, regular random graphs~\cite{Per_RRG_17}, the configuration model with finite-variance degrees~\cite{Svante_CM_fin_var_19}, infinite-variance and finite-mean degrees~\cite{Remco_Mia_WTA,van2023universalwinner} (see also the work of Baroni, Van der Hofstad, and Komj\'athy on competing FPP with deterministic edge-weights~\cite{BarHofKom15,BarHofKom15.2}),  and infinite-mean degrees~\cite{deijfen_vanderhofstad_sfragara_2023}. The preprint \cite{van2023universalwinner} shows a universal non-coexistence result for CMs with finite-mean but infinite-variance degrees under general assumptions. The literature on first-passage percolation processes with a single type on random graphs extends to many other models as well. This includes, but does not limit itself, to the complete graph~\cite{Jans99,AddBrouLug10,BhaHof17,HofHoogMie06,HoogMie08,BhaHof12,BhaHofHoog13,EckGoodHofNar13,EckGoodHofNar15.1,EckGoodHofNar15.2}, the Erd{\H o}s-R\'enyi random graph~\cite{BhaHofHoog11,HofHoogMie01,HofHoogMieg02}, the configuration model~\cite{AdrKom18,BarHofKom17,BarHofKom19,BhaHofHoog10.2,BhaHofHoog17}, preferential attachment models~\cite{JorKom20,JorKom22}, and spatial models such as scale-free percolation and geometric inhomogeneous random graphs~\cite{HofKom17,KomLod20,ColLimHinJahVal23}. Noteworthy is the recent work of Komj\'athy et al.~\cite{KomLapLen21,KomLapLenSchal23_1,KomLapLenSchal23_2}, who extensively study a degree-penalised variation of first-passage percolation on geometric inhomogeneous random graphs. In this model, edges incident to high-degree vertices typically have higher transmission times. This is, at least in spirit, similar to the spatial penalisation in LRFPP, though the underlying random graph model is very different.

\subsubsection{Main contribution}\label{sec:main_cont}
In many of the works mentioned above, both for single-type FPP processes and competition processes, approximating first-passage percolation with branching processes is a very useful and commonly used technique. We highlight here, in particular, the work of Bhamidi and Van der Hofstad on FPP on the complete graph~\cite{BhaHof12}, where a precise coupling between an FPP process on the complete graph and a CMJ branching process is provided. The main contribution of the current work is a variation of this coupling, in which we couple the infection processes to branching random walks. Here we also provides a more in-depth analysis of the defect of the coupling, where we can control the size of the defect further into the process. This analysis is crucial to many of the proofs in this paper. 

In~\cite{BhaHof12}, it is proved that the coupling between the FPP process and the CMJ branching process approximation is \emph{exact} up to a certain time $T$, with high probability. This result allows the authors to obtain precise asymptotic results for the typical weighted distance in the complete graph, as well as the hopcount (the number of edges on the shortest edge-weighted path). The analysis of the variant of this coupling carried out in the present work, additionally identifies the \emph{growth rate} of the coupling defect. That is, when the coupling is no longer exact, we control the error between the infection processes and the branching random walk approximation over time. In particular, we show that, up to some time $T'\gg T$, this error is negligible compared to the size of the entire branching random walks. This allows us to use the coupling to a higher degree and obtain useful quantitative bounds, much further into the competition process than previously possible.

The work of Ahlberg, Deijfen, and Janson on competing FPP on the configuration model with finite-variance degrees~\cite{Svante_CM_fin_var_19} is another example where branching-process approximations up to the time that the approximations are exact, are used. From that point onwards, their analysis shows that the competition dynamics are essentially deterministic, in the sense that the ratio of the sizes of the two competing infections stays roughly the same. These results are harder to prove for the model studied here, due to the spatial embedding of the vertices, even when the spatial dependence is weak (long-range parameters $\alpha_\square$ close to zero), see also Open Problem~\ref{prob:linear_initial_sets} below. However, our improved control of the coupling to branching processes still allows us to obtain the desired results, despite not being able to solve Problem~\ref{prob:linear_initial_sets}.

Finally, in most of the related literature discussed above, model parameters are fixed and independent of $n$. This generally yields only the possibilities $c_n\equiv 0$ for which coexistence occurs, or $c_n=\lambda \log n$ for some $\lambda\neq 0$ for which absence of coexistence, in the sense of Case $(iii)$ of Theorem~\ref{thrm:nocoex}, occurs. Allowing the model parameters to depend on $n$ uncovers more subtle behaviour, as in Cases $(i)$, $(ii)$, and $(iv)$ in Theorem~\ref{thrm:nocoex}, but also coexistence for which $c_n$ is not equal to zero, but is only bounded. 

\subsection{Discussion}

Let us discuss a few examples in this section by taking some particular choices of $\alpha_\ominus(n)$, $\alpha_\oplus(n)$, and $\lambda(n)$, for which we can apply Theorems~\ref{thrm:coex} and~\ref{thrm:nocoex}.

\begin{example}[Equal and fixed long-range parameters.]\label{ex:fixed} Take $\alpha_\ominus(n)=\alpha_\oplus(n)=\alpha \in [0,d)$. In this case, $Z_n=\lambda(n)$ and so $c_n=(\lambda(n)-1)\log{n}$. In this case, we observe absence of coexistence when $\log(n)(\lambda(n)-1)$ diverges and coexistence when $
\log(n)(
\lambda(n)-1)$ is bounded in $n$. Moreover,  
\begin{enumerate} 
    \item[$(i)$] $M_n=1$ with high probability when $\lambda(n)\gg \log n$.	\item[$(ii)$] $M_n$ converges to a geometric random variable when $\lambda(n)/\log{n}$ converges to a positive constant;
	\item[$(iii)$] $M_n=n^{(1+o(1))/\lambda(n)}$ when $\lambda(n)=o(\log{n})$; 
	\item[$(iv)$] $M_n=n\e^{-(1+o(1))(\lambda(n)-1)\log n}$ when $\lambda(n) \to 1$ but $\log (n)(\lambda(n)-1)$ diverges; 
\end{enumerate}
$M_n$ is linear in $n$ with high probability, when $(\lambda(n)-1)\log{n}$ is bounded in $n$. In particular, when $\lambda(n)=\lambda>1$ fixed, then $M_n=n^{1/\lambda+o(1)}$, while when $\lambda(n)=1$, $M_n$ is linear. We remark that the particular case where $\lambda(n) = \lambda \geq 1$ is fixed, is, in general, the setting of competing first-passage percolation models most often studied in the literature.
\end{example}

When $\alpha_\ominus=\alpha_\oplus=0$, the spatial embedding of the vertices plays no role, and the competition process takes place on the complete graph where the $\ominus$ (resp.\ $\oplus$) infection spreads through edges at rate $1$ (resp.\ $Z_n=\lambda(n)$). This setting can be exactly represented as an $n$-dependent  P\'olya urn, where the replacement matrix depends on $n$. In particular, adopting the setting of generalized P\'olya Urns of Janson \cite{Janson2006}, we consider a P\'olya Urn with two types of balls - $\ominus$ and $\oplus$ balls, where upon drawing a ball of some type, it is added back to the urn along with some additional balls according to the replacement matrix 
\[
\begin{pmatrix}
	1 & 0\\
	0 & Z_n 
\end{pmatrix}.
\]
This means, upon drawing a $\ominus$ ball (resp.\ $\oplus$ ball), one extra $\ominus$ ball is (resp.\ $Z_n$ extra $\oplus$ balls are) added back to the urn, along with the drawn ball. In this description, we implicitly think of $Z_n$ as being a positive integer, but that need not be, and in particular, the process has a description as a Markov process on the state space $\{(x,y)\in \mathbb{R}^2:x,y\geq 0, x+y>0\}$, see \cite[Remark $1.11$]{Janson2006}. However, following Janson \cite{Janson2006}, we use phrases like `$Z_n$ balls are drawn' for ease of understanding. 

Let $B^{\square}_n$ denote the number of Type $\square$ balls after $n$ draws, for $\square\in\{\ominus,\oplus\}$, where we initialise the urn with $B^{\ominus}_0=1, B^{\oplus}_0=Z_n$. It readily follows that $B_{n-2}^\ominus\overset d= N_n^\ominus$ and $B^\oplus_{n-2}/Z_n\overset d= N_n^\oplus$ when $\alpha_\ominus=\alpha_\oplus=0$. When $Z_n\equiv Z$ for some fixed $Z\geq 1$, results by Janson~\cite[Theorem $1.4$]{Janson2006} imply Theorems~\ref{thrm:coex} and~\ref{thrm:nocoex}. However, allowing $Z_n$ to depend on $n$ allows us to obtain more kinds of behaviour, in particular cases $(i)$, $(ii)$, and $(iv)$ in Theorem~\ref{thrm:nocoex}, which cannot be observed when $Z_n$ is constant. The results of Janson~\cite[Theorem $1.4$]{Janson2006} do state more detailed distributional convergence of $B_n^\square$ (when appropriately rescaled) in the case that $Z_n\equiv Z$, which we believe to be true also when $\alpha_\ominus=\alpha_\oplus=\alpha\in[0,d)$ are fixed. We discuss this further in Conjecture~\ref{conj:sym_asym}.

\begin{example}[Converging long-range parameters] \label{ex:conv} Fix $\alpha_\ominus,\alpha_\oplus\in[0,d)$ and let $\alpha_\square(n)\to\alpha_\square$ as $n\to\infty$ for $\square\in
\{\ominus,\oplus\}$. Let $\|\cdot\|_p$ denote the $\ell^p$-norm on the torus for $p\in[1,\infty)\cup\{\infty\}$. The sum $\sum_{x \in \T_n^d, x\neq 0}\|x\|_p^{-\alpha_\square(n)}$ for $\alpha \in [0,d)$, and $\square\in\{\ominus,\oplus\}$ equals $(R+o(1))n^{1-\alpha_\square(n)/d}$, where $R=R(\alpha_\square(n), d,p)$ depends on $\alpha_\square(n)$, the dimension $d$, and $p$, see \cite[Appendix]{HofLod23}. Moreover, $R(\alpha_\square(n),d,p)\to R(\alpha_\square,d,p)$ as $n\to\infty$. We thus obtain
\be \label{eq:Zncase}
Z_n=\lambda(n)\Big(\frac{R(\alpha_\oplus,d,p)}{R(\alpha_\ominus,d,p)}+o(1)\Big)n^{(\alpha_\ominus(n)-\alpha_\oplus(n))/d}.
\ee 
We can again, according to the criteria in Theorems~\ref{thrm:coex} and~\ref{thrm:nocoex}, determine whether coexistence or absence of coexistence occurs, and in the latter case, what the asymptotic size of the losing type is. This depends on how large/small $\lambda(n)$ is, whether $\alpha_\ominus(n)$ and $\alpha_\oplus(n)$ converge to different limits or to the same limit and at what rate, as well as the rate of convergence of $R(\alpha_\square(n),d,p)$ to $R(\alpha_\square,d,p)$. 
\end{example}

\subsubsection{Conjectures and Open problems}

As stated in Example~\ref{ex:fixed}, we believe that stronger distributional results can be achieved when coexistence occurs and in Case $(iii)$ of absence of coexistence in Theorem~\ref{thrm:nocoex}. 

\begin{conjecture}[Distributional convergence for (absence of) coexistence]\label{conj:sym_asym}
	Consider the same definitions and notation as in Theorem~\ref{thrm:nocoex} and let $E_1,E_2$ be two i.i.d.\ rate-one exponential random variables. If $c_n=(c+o(1))\log n$ for some $c\in\R$ and $|c_n|\to\infty$, then 
    \begin{align} 
    M_n n^{-\min\{Z_n,1/Z_n\}}&\toindis E_1 E_2^{-1/(|c|+1)}.
    \intertext{If, instead, $\lim_{n\to\infty}c_n=c\in\R$ exists, then} 
    \frac{M_n}{n} &\toindis  \frac{E_1}{E_1+E_2\e^c}.
    \end{align}
\end{conjecture}

\noindent In fact, we are even unable to show the following weaker result when $c_n\to c$, which we leave as an open problem:

\begin{problem}[Distributional asymmetry of the final sizes]\label{prob:sym_asym}
	Consider the same definitions and notations as in Theorem~\ref{thrm:nocoex}. If $\lim_{n\to\infty}c_n=c\in\R$ exists,
	then 
    \be 
    \lim_{n\to\infty}\mathbb P(N_n^{\ominus}<N_n^{\oplus})=\frac{1}{1+\e^{-c}}.
    \ee
\end{problem}

\noindent It is clear that Conjecture~\ref{conj:sym_asym} implies Problem~\ref{prob:sym_asym}. Towards proving the third part of Conjecture~\ref{conj:sym_asym}, by using a branching random walk approximation, for sufficiently small $\eps>0$, we can check that for $t_{\eps}:=\log(\eps n)$ we have that $|\cN_n^{\ominus}(t_{\eps})|$ and $|\cN_n^{\oplus}(t_{\eps})|$ are roughly $E_1 \eps n$ and $E_2 \eps \e^{c} n$ respectively, where $E_1, E_2$ are two i.i.d.\ rate-one Exponential random variables. Note that this implies that not all vertices of the torus are infected by time $t_{\eps}$ with probability close to $1$, for  small $\eps$, although both the infections have reached linearly many vertices. We were unable to apply similar techniques as used in e.g.~\cite{Svante_CM_fin_var_19}, to show that, at this stage, the ratios of the two infections remain roughly unchanged until the end of the process, due to the spatial embedding and long-range effects of the competition process. We leave this as an open problem:

\begin{problem}\label{prob:linear_initial_sets}
	Consider the same definitions and notations as in Theorem~\ref{thrm:nocoex}. Assume $c:=\lim_{n \to \infty} c_n$ exists. Take $A,B \subset \T_n^d$ with $|A|=\floor{an}$ and $|B|=\floor{bn}$ for $a,b>0$ with $A \cup B \neq \T_n^d$ and $A \cap B = \emptyset$, and set $\cN_n^{\ominus}(0)=A$ and $\cN_n^{\oplus}(0)=B$. Then $\frac{N_n^{\ominus}}{n} \toinp \frac{a}{a+b}$, and $\frac{N_n^{\oplus}}{n} \toinp \frac{b }{a+b}$.
\end{problem}

\noindent This problem, coupled with our earlier observation on the infection sizes up to time $t_{\eps}$, resolves Conjecture~\ref{conj:sym_asym} in the following way: we divide time into intervals $[0,t_{\eps})$ and $[t_{\eps}, T_{\mathrm{cov}}]$. At the end of the first time interval both the processes reach a linear size, with some arbitrary infection sets. Applying the Markov property at time $t_{\eps}$ and using problem~\ref{prob:linear_initial_sets} shows that $N_n^{\ominus}/n$ and $N_n^{\oplus}/n$ converge in distribution to $\frac{E_1}{E_1+\e^c E_2}$ and $\frac{e^c E_2}{E_1+\e^c E_2}$, respectively. Using P\'olya urns one can check Problem~\ref{prob:linear_initial_sets} is true for $\alpha_\ominus=\alpha_\oplus=0$, as discussed in Example~\ref{ex:fixed}. When the long-range parameters are bounded away from $0$, we were unable to resolve this problem.

\textbf{Long-range parameters approaching $d$.} \ \ \ 
The main assumption that we make is that $\sup_{n\in\N} \alpha_\ominus(n), \sup_{n\in\N} \alpha_\oplus(n)< d$. This is necessary for parts of the analysis carried out in the paper. When  either $\alpha_\ominus{(n)}$ or $\alpha_\oplus{(n)}$ converges to $d$,  estimates of the quantities $R_n^\ominus$ and $R_n^\oplus$, as in~\eqref{eq:Rndef}, fail to hold, which play a crucial role in the analysis. The case when either or both of $\sup_{n\in\N} \alpha_\ominus(n)$ and $\sup_{n\in\N} \alpha_\oplus(n)$ equals $d$, remains open. In particular, we leave the following conjecture, which is a first step beyond $\alpha_\ominus,\alpha_\oplus\in[0,d)$ into more unconstrained choices for the values of $\alpha_\ominus,\alpha_\oplus\geq 0$: 

\begin{conjecture}[(Absence of) coexistence for long-range parameters equal to $d$]
	Consider the case when $\alpha_\ominus(n)=\alpha_\oplus(n)=d$ for all $n\in\N$, and $\lambda(n)=\lambda>0$ a fixed constant. Then, coexistence occurs if and only if $\lambda=1$.
\end{conjecture}

\noindent The analysis of several crucial technical tools that we exploit in this paper does not carry through  when the long-range parameters are exactly equal to $d$. In particular, we are unable to prove results equivalent to those in Proposition~\ref{prop:gen_size_artificial_CTBP_tail_bound} and Lemma~\ref{lemma:phase2} (though we expect them to be true), both when we expect coexistence and the absence of it to hold. Here, our approximations of the rate at which infected vertices infect the remaining uninfected vertices seems to be too crude.\\

\textbf{On powers of exponentials.}\ \ \ 
In the model definition, a vertex $v$, infected by the $\square$ infection, spreads to an uninfected vertex $u$ at rate $(R_n^\ominus)^{-1}\lambda_\square \|u-v\|^{-\alpha_\square}$, with $\square\in\{\ominus, \oplus\}$. This is equivalent to saying that it takes the $\square$ infection $(R_n^\ominus)^{-1}\lambda_\square \|u-v\|^{-\alpha_\square} E_{u,v}$ amount of time to spread from $v$ to $u$, where $(E_{u,v})_{u\neq v\in \T_n^d}$ is a family of i.i.d.\ rate-one exponential random variables.  We expect similar results as in Theorems\ref{thrm:coex} and~\ref{thrm:nocoex} to hold when the random variables $E_{u,v}$ are raised to some power $\gamma>0$. The coupling between the infection processes and branching random walks that is used in this work would become more involved, however, along the lines of the coupling used in~\cite{BhaHof12}. In this situation, following~\cite[Remark 1.9]{SC_PD_LRFPP}, one would have to assume $\sup_{n\in\N}\alpha_\ominus(n),\sup_{n\in\N}\alpha_\oplus(n)<d/\gamma$ (instead being smaller than $d$, as in Assumption~\ref{ass:par}). We leave this as an open problem:
\begin{problem}
	Prove results similar to Theorems~\ref{thrm:coex} and~\ref{thrm:nocoex} when the time for the $\square $ infection to spread from $v$ to an uninfected vertex $u$ equals $(R_n^\ominus)^{-1}\lambda_\square \|u-v\|^{-\alpha_\square} E_{u,v}^\gamma$, for some $\gamma>0$, where the $(E_{u,v})_{u\neq v \in \T_n^d}$ is an i.i.d.\ collection of rate-one exponential random variables. In particular, do the phase transitions of the scale $c_n$ stay in the same order, i.e.\ orders of $(\log{n})^2, \log{n}$, and constant? Do analogues of Conjecture~\ref{conj:sym_asym} or Problem~\ref{prob:sym_asym} hold?
\end{problem}

\textbf{Competition on the long-range percolation cluster.} \ \ \ 
\noindent Consider the \emph{long-range percolation} graph constructed on the torus $\T_n^d$ as follows. We include all the nearest-neighbour edges. Further, for pairs $x,y \in \T_n^d$ that are non-nearest neighbours, we include the edge $\{x,y\}$ independently of all other edges with probability $ \|x-y\|^{-\alpha}$, for some fixed $\alpha \in (0,d)$ and where $\|\cdot\|$ is the torus norm. Note that this graph is always connected as it contains $\T_n^d$ as a subgraph. Consider a standard competing FPP process on this graph started from two uniformly chosen vertices $v^{\ominus}$ and $v^{\oplus}$, and let us call these the $\ominus$ and $\oplus$ infections, respectively, as before, where the $\ominus$ infection spreads at rate $1$ and the $\oplus$ at rate $\lambda=\lambda(n)$ along edges. The model studied here can be viewed as a mean-field version of this process, so that we  expect similar results to Theorems~\ref{thrm:coex} and~\ref{thrm:nocoex} and to hold for the sizes $N_n^{\ominus}$ and $N_n^{\oplus}$ of the two infections when all the vertices have been infected.

\begin{conjecture}
	Let $\lambda>0$, and consider the competing FPP on the long-range percolation cluster with $\alpha<d$ as described above. Then,
	\begin{itemize}
		\item If $\lambda=1$, then coexistence occurs. 
        \item If $\lambda\neq 1$, then $\log(M_n)/\log n\toinp \min\{\lambda,1/\lambda\}$.
	\end{itemize}
\end{conjecture}

\noindent More interesting is the case $\lambda(n)=1+\frac{c_n}{\log{n}}$ for some sequence $c_n$. In particular, it would be interesting to see whether for this case one obtains asymptotic sizes similar to Theorems~\ref{thrm:coex} and~\ref{thrm:nocoex}.

\begin{problem}
	Consider the competition model on the long-range percolation cluster as described above, where the $\ominus$ and $\oplus$ infections spread along edges at rate $1$ and $\lambda(n)=1+c_n/\log{n}$, respectively. Is the asymptotic behaviour of $M_n$ similar to 
	that in Theorems~\ref{thrm:coex} and~\ref{thrm:nocoex}, when $c_n$ behaves in the different ways as described in those theorems? Further, do analogues of Conjecture~\ref{conj:sym_asym} (with possibly different limits) or Problem~\ref{prob:sym_asym} hold?  
\end{problem}

\textbf{Organisation of the paper.}\ \ \ 
The remainder of the paper is organized as follows. In Section~\ref{sec:prelims} we state and prove several preliminary results. In Section~\ref{sec:coupling}, we a coupling of the infection growths to branching random walks. In particular, Section~\ref{sec:defect} contains the main technical contribution, formalizing that the branching random walk coupling can be controlled up to a certain time. In Section~\ref{sec:ub} we prove the upper bounds on the size of the $\ominus$ infection that are claimed in Theorem~\ref{thrm:nocoex}. In particular, these upper bounds imply when absence of coexistence occurs. We prove matching lower bounds on the size of the $\ominus$ infection and complete the proofs of the main theorems in Section~\ref{sec:nocoex}.

\section{Preliminary results}\label{sec:prelims}

In this section we collect a number of minor technical preliminary results that we need throughout the paper. They consist of estimating the quantities $R_n^\ominus$ and $R_n^\oplus$, as in~\eqref{eq:Rndef}, as well as some basic well-known results on the growth rate of branching processes. 

\subsection{Rate estimation}

Recall the quantities $R_n^\ominus$ and $R_n^\oplus$ from~\eqref{eq:Rndef} (with $\lambda_\oplus=\lambda(n)$, see Remark~\ref{rem:wlog}). In this section we provide quantitative bounds and asymptotic results for these quantities. To this end, we define, with $\alpha=\alpha(n)\in[0,d)$,
\be\label{eq:Rn}
R_n(\alpha):=\sum_{\substack{v\in\T_n^d\\ v\neq \mathbf{0}}}\|v\|^{-\alpha}.
\ee 
For values of $n$ such that the radius of the torus $n^{1/d}$ is not an integer, $R_n(\alpha)$ is not well-defined. More generally, set $M(n):=\inf\{m\in \N: m^d\geq n\}$ and  let $v_1,\ldots, v_M$ be the vertices in $\T_M^d$ in increasing order with respect to their distance to the origin (using the norm $\|\cdot\|$). We then set 
\be 
R_n(\alpha):=\sum_{i=1}^n \|v_i\|^{-\alpha}.
\ee 
It follows from~\cite[Appendix]{HofLod23} that, by switching from summation to integration and applying appropriate variable substitutions, if $\alpha$ is fixed and independent of $n$, then $R_n(\alpha)=(R+o(1))n^{1-\alpha/d}$ as $n\to\infty$, where 
\be \label{eq:R}
R=R(\alpha,d,p)=2^\alpha \int_{[0,1]^d} \|x\|^{-\alpha}\,\dd x=2^d \int_{[0,\frac12]^d}\|x\|^{-\alpha}\,\dd x.
\ee 
{Here, $R$ depends on $\alpha, d$, and the $\ell^p$-norm that is used, for $p\in[1,\infty)\cup\{\infty\}$. As we are only interested in the dependence on $\alpha$ in what follows, we omit the arguments $d$ and $p$.} The final representation on the right-hand side follows from a variable substitution. This representation, however, allows us to directly conclude that $R(\alpha)$ is increasing in $\alpha$, regardless of which $\ell^p$-norm is used. It is also readily verified that the asymptotic result $R_n=(R+o(1))n^{1-\alpha/d}$ also holds when $\alpha=\alpha(n)\in[0,d)$ depends on $n$, as long as $\sup_{n\in\N}\alpha(n)<d$ is satisfied. Moreover, it holds along any diverging sub-sequence $(j(n))_{n\in\N}$, in the sense that
\be \label{eq:Ralphan}
R_{j(n)}(\alpha(n))=(R(\alpha(n))+o(1))j(n)^{1-\alpha(n)/d}, \qquad\text{as }n\to\infty,
\ee 
where $R(\alpha(n))$ as in~\eqref{eq:R} now also depends on $n$. We have the following useful bounds for $R_j(\alpha(n))$ for all $j\in\N$, as well as bounds for the sum of the rates between disjoint sets  $A,B\subset \T_n^d$, based on $R_n$:

\begin{lemma}\label{lemma:ratesupinf}
	Let $\alpha=\alpha(n)\in[0,d)$ be such that $\sup_{n\in\N}\alpha(n)<d$. There exist constants $0<\underline c<\overline c<\infty$, such that, uniformly over $j\in\N$ and $n\in\N$, 
	\be
	\underline c\leq \frac{R_j(\alpha(n))}{j^{1-\alpha(n)/d}}\leq \overline c.
	\ee 
\end{lemma}

\begin{proof}
	Set $\overline \alpha:=\sup_{n\in\N}\alpha(n)$. It follows from~\eqref{eq:Ralphan} that $R_j(\alpha(n))=(R(\alpha(n))+o(1))j^{1-\alpha(n)/d}$ as $j\to \infty $, where $j$ can either be independent of $n$ and tend to infinity, or $j$ can depend on $n$ such that $n\to\infty$ and $j(n)\to \infty$. Then, since $R$ is increasing in $\alpha$, we bound $1=R(0)\leq R(\alpha(n))\leq R(\overline\alpha)<\infty$, uniformly over the sequence $(\alpha(n))_{n\in\N}$. As a result, taking $\underline c$ sufficiently small and $\overline c$ sufficiently large, we obtain the desired result. 
\end{proof}

\begin{lemma}\label{lemma:rateest}
	Let $\alpha\in[0,d)$, fix $A,B\subset \T_n^d$ such that $A\cap B=\emptyset$, and write $A^c:=\T_n^d\setminus A$ for the complement of $A$. Then,
	\begin{align}
		\sum_{v\in A}\sum_{u\in B} \|u-v\|^{-\alpha}&\leq \min\{|A|R_{|B|}(\alpha),|B|R_{|A|}(\alpha)\}, \\
		\intertext{and}
		\sum_{v\in A}\sum_{u\in B}\|u-v\|^{-\alpha}&\geq \max\{|A|(R_n(\alpha)-R_{|B^c|}(\alpha)),|B|(R_n(\alpha)-R_{|A^c|}(\alpha))\}.
	\end{align}
	In particular, for $B=A^c$,
	\begin{align}
		\sum_{v\in A}\sum_{u\in A^c} \|u-v\|^{-\alpha}&\leq \min\{|A|R_{|A^c|}(\alpha),|A^c|R_{|A|}(\alpha)\}\leq \min\{|A|,|A|^c\}R_n(\alpha), \\
		\intertext{and}
		\sum_{v\in A}\sum_{u\in A^c}\|u-v\|^{-\alpha}&\geq \max\{|A|(R_n(\alpha)-R_{|A|}(\alpha)),|A^c|(R_n(\alpha)-R_{|A^c|}(\alpha))\}.
	\end{align}
\end{lemma}

\begin{proof}
	For the upper bound we observe that for each $v\in A$, the sum of the inverse distances from $v$ to all vertices $u\in B$ is always smaller than the sum of the inverted distances of the $|B|$ many closest vertices to $v$. Hence, 
	\be
	\sum_{v\in A}\sum_{u\in B}\|u-v\|^{-\alpha}\leq \sum_{v\in A}R_{|B|}(\alpha)=|A|R_{|B|}(\alpha). 
	\ee 
	By symmetry the upper bound $|B|R_{|A|}(\alpha)$ also holds. 
	The final upper bound (with $B=A^c$) directly follows since $R_k(\alpha)$ is increasing in $k$. 
	
	For the lower bound we observe that for each $v\in A$,
	\be 
	\sum_{u\in B}\|u-v\|^{-\alpha}=R_n(\alpha)-\sum_{\substack{u\in B^c\\ u\neq v}}\|u-v\|^{-\alpha}\geq R_n(\alpha)-R_{|B^c|}(\alpha),  
	\ee 
	where we again use that the sum over $u\in B^c, u\neq v$, is at most the sum of the inverted distances of the $|B^c|$ many closest vertices to $v$. Since the lower bound does not depend on $v$, we obtain the first argument of the maximum. Symmetry again concludes the proof.
\end{proof}

\subsection{Elementary results on branching processes}

For $\sigma>0$ we define a  rate-$\sigma$ \emph{Yule process} $(\BP(t))_{t\geq 0}$ as a branching process in continuous time, started from a root $\varnothing$ at time $t=0$. Each individual born into the process $($including $\varnothing)$ has an associated i.i.d.\ exponential clock of rate $\sigma$. Every time the clock of an individual rings, it gives birth to a child $($receiving its own associated exponential clock$)$ and resets its clock. We refer to the start of Section~\ref{sec:coupling} for a more formal definition, which we omit for now and focus rather on the following result.

\begin{lemma} \label{lem:size_CTBP}
	Fix $\sigma>0$ and consider a rate-$\sigma $ Yule process $(\BP(t))_{t \geq 0}$.
	\begin{enumerate}[label=\arabic*.]
		\item For any $t\geq 0$, it holds that $|\BP(t)|\sim\mathrm{Geo}(\e^{-\sigma t})$ (see~\eqref{eq:geo} for the distribution's definition) and that $|\BP(t)|\e^{-t}\toas E$, with $E$ a rate one exponential random variable.
		\item  We have \begin{align*}
			\lim_{A \to \infty} \Prob{\{\sup_{t > 0}|\BP(t)|\e^{-\sigma t}\leq A\}\cap\{\inf_{t>0}|\BP(t)|\e^{-\sigma t}\geq A^{-1}\}} = 1.
		\end{align*} 
	\end{enumerate}
	
\end{lemma}

\noindent We do not prove Lemma~\ref{lem:size_CTBP} (1); it is quite standard and we refer the reader to~\cite[Section 2.5]{Nor98} for a proof (see also~\cite[Lemma $5.2$]{BanBha22} for a summary of the necessary results in Section $2.5$ of~\cite{Nor98}).

\begin{proof}[Proof of Lemma~\ref{lem:size_CTBP} (2)]	
	Since a Yule process of rate $\sigma$, evaluated at time $t$, is equal in distribution to a Yule process of rate $1$ evaluated at time $\sigma t$, we can without of generality set $\sigma=1$ in the proof. We have from $(1)$ that 
	\be \label{eq:bpas}
	|\BP(t)|\e^{-t}\toas E,
	\ee 
	and that $\E{|\BP(t)|}=\e^t$ for any $t\geq 0$. We use both results in the following. First, we can bound 
	\be 
	\P{\sup_{t\geq 0}|\BP(t)|\e^{-t}> A}\leq \P{\sup_{t\leq \log(A)/2}\!\!\!\!\!|\BP(t)|\e^{-t}>A}+\P{\sup_{t\geq \log(A)/2}\!\!\!\!\!|\BP(t)|\e^{-t}>A}. 
	\ee 
	For the first probability, we bound the exponential term from above by $1$ and observe that $|\BP(t)|$ is monotone increasing in $t$. As a result, using Markov's inequality yields 
	\be 
	\P{\sup_{t\leq \log(A)/2}|\BP(t)|\e^{-t}>A}\leq A^{-1}\E{|\BP(\log(A)/2)|}=A^{-1/2}. 
	\ee 
	For the second probability, we use the almost sure convergence in~\eqref{eq:bpas} to obtain the upper bound
	\be \ba 
	\mathbb P\Bigg({}&\Big|\sup_{t\geq \log(A)/2}\!\!\!\!\!|\BP(t)|\e^{-t}-E\Big|>1\Bigg)+\P{E>A-1}\\ 
	&=\P{\Big|\sup_{t\geq \log(A)/2}\!\!\!\!\!|\BP(t)|\e^{-t}-E\Big|>1}+\e^{1-A}.
	\ea \ee 
	The first term vanishes with $A$ by~\eqref{eq:bpas}. Thus, for any $\eps>0$ we can choose $A$ sufficiently large such that the probability on the right-hand side is at most $\eps/6$, as well as $\e^{1-A}<\eps/6$ and $A^{-1/2}<\eps/6$, to arrive at the desired bound
	\be 
	\P{\sup_{t\geq 0}|\BP(t)|\e^{-t}> A}<\eps/2.
	\ee 
	For the infimum, we first observe that $|\BP(t)|\e^{-t}\geq \e^{-t}>A^{-1}$ for any $t<\log A$. Hence, 
	\be 
	\P{\inf_{t\geq 0}|\BP(t)|\e^{-t}<A^{-1}}=\P{\inf_{t\geq \log A}|\BP(t)|\e^{-t}<A^{-1}}.
	\ee 
	Similar calculations as for the supremum yield that for any $\eps>0$, we can choose $A=A(\eps)$ sufficiently large such that the right-hand side is at most $\eps/2$. Combined with the bound for the supremum and since $\eps$ is arbitrary, we obtain the desired result.
\end{proof}

\section{Coupling the LRC process to branching random walks}\label{sec:coupling}

In this section we introduce and analyse a coupling between the LRC process and two independent continuous-time branching random walk (CTBRW) processes on the torus $\T_n^d$. To define the branching random walk, we use the perspective of random walks indexed by a branching process. Let
\be 
\cU_\infty:=\{\varnothing\}\cup\bigcup_{n\in\N}\mathbb N^n
\ee 
denote the Ulam-Harris tree. Here, $\varnothing$ denotes the root, and an individual $u_1\cdots u_k$ for some $k\in\N$ and $u_1,\ldots, u_k\in\N$ is recursively viewed as the $u_k^{\mathrm{th}}$ child of $u_1\cdots u_{k-1}$ (for $k=1$ we use the convention that $u_1\cdots u_{k-1}=\varnothing$). In particular, $u_1$ is the $u_1^{\mathrm{th}}$ child of $\varnothing$. For $u=u_1\ldots u_k\in\cU_\infty\setminus\{\varnothing\}$ for some $k\in\N$ and $u_1,\ldots u_k$, we let $a(u):=u_1\ldots u_{k-1}$ denote the parent (or ancestor) of $u$. 

We define a CTBRW $(\cX_n(t))_{t\geq 0}$ on the state space $\cU_\infty \times \T_n^d$. For $\sigma>0$, we first define a Yule process $(\BP_\sigma(t))_{t\geq 0}$ as a random subtree of $\cU_\infty$ as follows. For $u\in\cU_\infty$, let $(E_i^u)_{i\in\N}$ be i.i.d.\ exponential random variables with rate $\sigma$, which are also mutually independent among the $u\in\cU_\infty$. We initialise $\BP_\sigma(0)$ as $\{\varnothing\}$ and for $t\geq 0$, we set 
\be 
\BP_\sigma(t):=\bigg\{u_1\cdots u_k\in\cU_\infty\bigg| \sum_{j=1}^k \sum_{i=1}^{u_j}E^{u_1\cdots u_{j-1}}_i\leq t\bigg\}.
\ee 
Then, let $p_n=(p_n(v))_{v\in \T_n^d\setminus \{\mathbf 0\}}$ be a probability distribution on $\T_n^d\setminus\{\mathbf 0\}$, called the \emph{displacement distribution}, and let $(X_{a(u),u})_{u\in \cU_\infty\setminus\{\varnothing\}}$ be a family of i.i.d.\ random variables with distribution $p_n$. For $t\geq 0$ and $u\in\BP(t)$ we assign a location $x_u\in\T_n^d$ to the individual $u$. For a given $v\in \T_n^d$, we set $x_\varnothing:=v$ and 
\be
x_u:=x_{a(u)}+X_{a(u),u},\qquad \text{for }u\in \BP_\sigma(t)\setminus\{\varnothing\}.
\ee 
We can then define the CTBRW $(\cX_n(t))_{t\geq 0}$ as 
\be
\cX_n(t):=\{(u,x_u): u\in \BP_\sigma(t)\}.
\ee  
Intuitively, we think of the CTBRW $\cX_n(t)$ as particles on $\T_n^d$ that independently produce children at rate $\sigma$, who then independently jump to a different location according to the displacement distribution $p_n$ (translated by the position of their parent).

\subsection{Coupling definition} \label{sec:coupl} Recall from Section \ref{sec:model_def} the definition of the LRC process as a Markov process $(\cN_n^\ominus(t),\cN_n^\oplus(t))_{t \geq 0}$ on the state space $\Omega:=\{(A,B):A,B \subset \T_n^d,A\cap B=\varnothing\}$, with initial state $(\{v^\ominus\},\{v^\oplus\})$ for two distinct vertices $v^\ominus,v^\oplus \in\T_n^d$. Consider two independent CTBRWs $(\cX_n^\square(t))_{t \geq 0}$ with an initial particle at $v^\square$, for $\square \in \{\ominus,\oplus\}$. We use the notation introduced above, with superscripts $\ominus $ and $\oplus$ to distinguish between the two CTBRWs and we say that particles in $\cX_n^\square(t)$ are type $\square$ particles. Here, the birth rate of the Yule process $\BP^\ominus$ is $1$ (i.e.\ $\sigma=1$) and the birth rate of $\BP^\oplus$ is $Z_n$ (i.e.\ $\sigma=Z_n$). The displacement distributions $p_n^\ominus$ and $p_n^\oplus$ are 
\be 
p^\square_n(v):= \frac{1}{R_n^\square} \lambda_\square\| v\|^{-\alpha_\square}\qquad \text{for }\square \in \{\ominus,\oplus\}\text{ and } v\in \T_n^d\setminus\{\mathbf 0\},
\ee 
where we recall that $\lambda_\ominus=1$ and $\lambda_\oplus=\lambda(n)$ due to Remark~\ref{rem:wlog}. In this section we provide a coupling between the LRC process and the two CTBRWs $\cX_n^\ominus$ and $\cX_n^\oplus$.

\subsubsection{Original and artificial particles}
    Consider the two independent branching random walks $(\cX_n^\ominus(t))_{t \geq 0}$ and $(\cX_n^\oplus(t))_{t \geq 0}$. We define two marked branching random walks, where each particle receives a mark $\mathscr A$ or $\mathscr O$. We call such marked particles \emph{artificial} or \emph{original}, respectively. For a particle $u$ in $\cX_n^\square(t)$, we let $m^\square_u\in\{\mathscr A,\mathscr O\}$ denote the mark of $u$. We abuse notation to now define $(\cX_n^\square(t))_{t\geq 0}$ on $\cU_\infty \times \T_n^d\times \{\mathscr A,\mathscr O\}$ as 
    \be
    \cX_n^\square(t):=\{(u,x_u,m^\square_u): u\in \BP^\square(t)\}\qquad\text{for }\square \in \{\ominus,\oplus\}.
    \ee 
    The two initial particles of $\cX_n^\ominus$ and $\cX_n^\oplus$, located at $v^\ominus$ and $v^\oplus$ at time $t=0$, are marked $\mathscr O$ (where $v^\ominus\neq v^\oplus$). Suppose that a transition takes place in either $(\cX_n^\ominus(t))_{t\geq 0}$ or $(\cX_n^\oplus(t))_{t\geq0}$ at some time $t>0$, due to a new particle $u$ of type $\square\in\{\ominus,\oplus\}$ being born at $x_u\in \T_n^d$. We mark this new particle as
    \begin{itemize}
        \item $m_u^\square=\mathscr A$ if its parent $a(u)$ has mark $\mathscr A$, or if there is at least one particle marked $\mathscr O$ (of type $\ominus$ or $\oplus$) located at $x_u$ at time $t^-$;
        \item $m_u^\square=\mathscr O$ if its parent $a(u)$ is marked $\mathscr O$ and there are no particles marked $\mathscr O$ (of type $\ominus$ nor of type $\oplus$) located at $x_u$ at time $t^-$. 
    \end{itemize} 
    
    \noindent We can thus speak of original and artificial particles of type $\square\in\{\ominus,\oplus\}$. 

    For any $t \geq 0$ and $\square \in \{\oplus, \ominus\}$, let 
    \be \label{eq:Ondef}
    \cO_n^{\square}(t):=\{v\in \T_n^d: \exists u \in \cU_\infty \text{ such that } (u,v,\mathscr O)\in \cX_n^\square(t)\}
    \ee 
    denote the locations of the original particles in $\cX_n^\square(t)$. The key result of this section is the following.
\begin{proposition}[The LRC process equals original particles in marked CTBRWs]\label{prop:coup_LRC_CTBRW}
    Recall the LRC process $(\cN_n^{\ominus}(t),\cN^{\oplus}_n(t))$, as defined in Section~\ref{sec:model_def}, and the original particle locations $(\cO_n^\ominus(t),\cO_n^\oplus(t))_{t\geq 0}$ of the marked CTBRWs, as in~\eqref{eq:Ondef}. We then have $(\cN_n^{\ominus}(t),\cN_n^{\oplus}(t))_{t\geq0}\overset \dd =(\cO_n^{\ominus}(t),\cO_n^{\oplus}(t))_{t\geq 0}$.
\end{proposition}
    
\begin{proof}
    The statement is clearly true when $t=0$. Assuming the statement is true for some $t\geq 0$, it suffices to show that the distribution of the next transition in both the tuples are same. Observe that, by definition of the processes $\cX^\square(t)$, given $(\cO^\ominus(t),\cO^\oplus(t))$, for $\square \in \{\ominus,\oplus\}$ the rate at which an original particle of type $\square$ located at $v$   gives birth to an original particle of type $\square$ at a location $u$ with no original particles yet, i.e.\ at some $u \in (\cO^\ominus(t)\cup \cO^\oplus(t))^c$, is $(R_n^\ominus)^{-1}\lambda_\ominus\|u-v\|^{-\alpha_\ominus}$ when $\square=\ominus$ and $(Z_n/R_n^\oplus) \lambda_\oplus\|u-v\|^{-\alpha_\oplus}=(R_n^\ominus)^{-1}\lambda_\oplus\|u-v\|^{-\alpha_\oplus}$ when $\square=\oplus$. This is the same rate at which a vertex $v\in \cN_n^\square(t)$ spreads to a vertex $u\in (\cN^\ominus(t)\cup \cN^\oplus(t))^c$, which yields the desired result.
\end{proof}

We observe that, whilst the two branching random walks $\cX_n^\ominus(t)$ and $\cX_n^\oplus(t)$ are independent, the processes $\cO_n^\ominus(t)$ and $\cO_n^\oplus(t)$ are not. Leveraging the distributional identity between the LRC process and the original particles of the two marked branching random walks, we can control the size of $\cN_n^\ominus(t)$ and $\cN_n^\oplus(t)$ via the independent processes $\cX_n^\ominus(t)$ and $\cX_n^\oplus(t)$, which are more amenable to analysis.

\begin{remark} 
In the remainder of the paper, we may use the notation $\cN^\square_n(t)$ and $\cO_n^\square(t)$ interchangeably, which in light of Proposition~\ref{prop:coup_LRC_CTBRW} we see as the same objects. Whenever we use the coupling introduced in this section we use $\cO_n^\square$ and when we use other arguments (e.g.\ discrete P\'olya urn type arguments) we use the notation $\cN_n^\square$. In either case, the notation should be clear from context.
\end{remark}

\subsection{Tail bounds on the coupling defect}\label{sec:defect}

We define, for $\square\in\{\ominus,\oplus\}$ and $t\geq0$,
\be\label{eq:Dn}
\cD_n^\square(t):=\{(u,x_u,m^\square_u)\in \cX_n^\square(t): m_u=\mathscr A\},
\ee
as the \emph{coupling defect} at time $t$. That is, $\cD_n^\square(t)$ denotes the set of all artificial particles of type $\square$ born up to time $t$, which do not have a corresponding `birth` in the LRC process. In this sub-section we control the growth of $\cD_n^\square$ over time, which provides quantitative bounds on how well the number of particles born in $(\cX_n^\ominus(t),\cX_n^\oplus(t))$ approximates the sizes of the infections $(\cN_n^\ominus(t),\cN_n^\oplus(t))$, which forms the main technical tool used throughout the proofs of Theorems~\ref{thrm:coex} and~\ref{thrm:nocoex}.

We remark that an analysis of a similar kind was done in \cite{BhaHof12}, up to a point when artificial particles first start appearing -- an event the authors of \cite{BhaHof12} term \emph{collisions} in their paper. At this point in time, the size of the defects are of constant order. We move significantly beyond this, by also controlling the size of the defects when they are of \emph{linear} order (with respect to the size of $\cX^\ominus_n(t)$ and $\cX^\oplus_n(t)$) and show that it is still `sufficiently small' to be of use for our purposes.

Recall that $a(u)$ denotes the parent of a $u\neq \varnothing$ in the Ulam-Harris tree $\cU_\infty$. For $\square\in\{\ominus,\oplus\}$ and $t\geq0$, we define
\be 
\cR^\square_n(t):= \{(u,x_u,\mathscr A)\in \cD^\square(t): m_{a(u)}^\square=\mathscr O\}.
\ee 
By the independence property of sub-trees of branching processes (and thus in particular of Yule processes), we can view the artificial particles born in $\cX_n^\square(t)$ as independent Yule processes rooted at the artificial particles in $\cR_n^\square(t)$, which we call \emph{artificial root} particles of type $\square$. We can use this perspective to characterise the size of $\cD_n^\square(t)$. 

We define a sequence of stopping times $(\tau_k^\square)_{k\in\N,\square\in\{\ominus,\oplus\}}$  as
\be\label{eq:tauk}
\tau^{\square}_k:=\inf\{t>0:|\cR_n^\square(t)|=k\}.
\ee
Furthermore, let $(\BP_j^{\square}(t))_{j\in\N,t\geq 0}$ be a collection of independent Yule processes, where the $\BP_j^\square$ are i.i.d.\ copies of $\BP^\square$ (as introduced at the start of Section~\ref{sec:coupl}) with rate $1$ if $\square=\ominus$ and with rate $Z_n$ if $\square=\oplus$. We use the convention that $|\BP_j^{\square}(a)|=0$ for any $j\in\N$ when $a\leq 0$. We can then write
\be\label{eq:C}
|\cD_n^\square(t)|\overset \dd=\sum_{j=1}^{\infty} \ind{\tau_j^{\square}\leq t}|\BP_j^{\square}(t-\tau_j^{\square})|.  
\ee
We use this distributional equality to prove the following bounds on the size of $\cD_n^\square(t)$. 

\begin{proposition}\label{prop:gen_size_artificial_CTBP_tail_bound}
	Fix $\eps>0$ and $\delta \in (0,\inf_{n\in\N}\min\{1-\alpha_\ominus(n)/d,1-\alpha_\oplus{(n)}/d\})$, and recall $Z_n$ from~\eqref{eq:Zn}. The following statements hold. 
    \begin{enumerate}
	    \item There exists $M=M(\eps,\delta)$ such that for all $n\in\N$, $m\geq M$, and $\square\in\{\ominus,\oplus\}$, \be \label{eq:Clin}
        \Prob{\Big|\cD_n^\square\Big(\frac{\log{n}-m}{Z_n}\Big)\Big|>n\e^{-m(1+\delta)}}<\varepsilon.
        \ee 
    \item Let $c_n=(Z_n-1)\log n$ be such that $c_n\to \infty$ with $n$ and let $\delta_n$ be such that $\limsup_{n\to\infty}\delta_n<1$ and $\delta_n\log n\to\infty $ with $n$. Then,
	\be \label{eq:Csublin}
	\Big|\cD_n^{\ominus}\Big(\frac{(1-\delta_n)\log n}{Z_n}\Big)\Big|\e^{-(1-\delta_n)\log(n)/Z_n}\toinp 0.
	\ee
	\end{enumerate}
\end{proposition}

\noindent To prove Proposition~\ref{prop:gen_size_artificial_CTBP_tail_bound}, we first define the $\sigma$-algebra 
\be \label{eq:F}
\mathcal{F}=\sigma\{(\cO_n^\ominus(s),\cO_n^\oplus(s)):0<s<\infty\},
\ee 
and the stopping times $(T_k)_{k=2}^n$ as 
\be \label{eq:Tkdef}
T_k:=\inf\{t>0: |\cO^\ominus_n(t)\cup \cO^\oplus_n(t)|=k\}\qquad \text{for }k\in\{2,\ldots,n\}.
\ee 
Clearly, $\cD_n^\square$ depends on $\cF$ and $(T_k)_{k=2}^n$, as artificial root particles in the sets $\cR_n^\square(t)$ are born from original particles. Towards proving Proposition~\ref{prop:gen_size_artificial_CTBP_tail_bound}, let us start with an expression for the expected value of $\cD_n^{\square}(t)$. Recall from Remark \ref{rem:wlog}, that we assume without loss of generality $\lambda_\ominus=1$ and $\lambda_\oplus=\lambda=\lambda(n)$.

\begin{lemma}\label{lemma:Cexp}
	Let $\cF$ be the $\sigma$-algebra in~\eqref{eq:F}, define $r^\square:=\ind{\square=\ominus}+\ind{\square=\oplus}\lambda$ and $s^\square:=\ind{\square=\ominus}+\ind{\square=\oplus}Z_n$. Then, for any $t\geq 0$ and $\square\in\{\ominus,\oplus\}$, 
	\be 
	\mathbb E\big[|\cD_n^\square(t)|\,\big|\, \cF\big]=\int_0^t \!\Bigg(\frac{1}{R_n^\ominus}\sum_{u \in \cO_n^\square(z)}\sum_{\substack{v \in \cO_n^\ominus(z)\cup \cO_n^\oplus(z)\\ v\neq u}}\!\!\!\!\!\!\!\!\!\!\!\!\!\!\!\!r^\square\|u-v\|^{-\alpha_\square}\! \Bigg)\e^{s^\square(b-z)}\, \dd z.
	\ee 
\end{lemma}

\begin{proof}
	We recall the sequence $(\tau_j^{\square})_{j\in \N}$ defined in~\eqref{eq:tauk}; the stopping times at which artificial root particles of type $\square$ are born. By conditioning on both $\cF$ and $(\tau_j^{\square})_{j\in\N}$, we can combine~\eqref{eq:C} with Lemma~\ref{lem:size_CTBP} to obtain 
	\be 
	\E{|\cD_n^{\square}(t)|\,\Big|\,\cF,(\tau_j^{\square})_{j\in\N}}=\sum_{j=1}^\infty \ind{\tau_j^{\square
    }\leq t}\e^{s^\square(t-\tau_j^\square)}.
	\ee 
	We now want to take the expected value with respect to the sequence $(\tau_j^{\square})_{j\in\N}$. We first define the function $f^\square_t(z):=\e^{s^\square(t-z)}$, and note that
	\be 
	\sum_{j=1}^\infty \ind{\tau_j^{\square}\leq t}\e^{s^\square(t-\tau_j^{\square})}=\int_0^tf^\square_t(z)\, \eta^{\square}(\dd z), \quad\text{where}\quad \eta^{\square}:=\sum_{j=1}^\infty \delta_{\tau_j^{\square}}, 
	\ee with $\delta$ a Dirac measure. Recall the sequence of stopping times $(T_k)_{k=2}^n$ from~\eqref{eq:Tkdef}, and let $T_{n+1}=\infty$. We observe that the sequence $(T_k)_{k=2}^{n+1}$ is measurable with respect to $\cF$. For any $k\in\{2,\ldots, n\}$, the point process $\eta^\square$, restricted to the interval $[T_k, T_{k+1})$, is distributed as a Poisson point process $\cP^\square_k$, restricted to the interval $[T_k,T_{k+1})$, with constant intensity 
	\be 
	\frac{1}{R_n^{\ominus}}\sum_{u \in \cO_n^{\square}(T_k)}\sum_{\substack{v \in \cO_n^{\ominus}(T_k)\cup \cO_n^{\oplus}(T_k)\\ v\neq u}}r^\square\|u-v\|^{-\alpha_\square}.
	\ee 
	Moreover, as $[0,T_3)=[T_2,T_3),[T_3,T_4),\ldots, [T_n,T_{n+1})=[T_n,\infty)$ partitions $[0,\infty)$, we can write
	\be \ba 
	\E{|\cD_n^\square(t)|\,\Big|\,\cF}&=\E{\int_0^tf^\square_t(z)\, \eta^{\square}(\dd z)\,\bigg|\, \cF}\\ 
	&=\sum_{k=2}^{n+1} \E{\int_{[0,t]\cap [T_k,T_{k+1})} f^\square_t(z)\, \eta^{\square}(\dd z)\,\bigg|\, \cF}\\
	&=\sum_{k=2}^{n+1} \E{\int_{[0,t]\cap [T_k,T_{k+1})} f^\square_t(z)\, \cP^{\square}_k(\dd z)\,\bigg|\, \cF}.\\
	\ea\ee 
	Campbell's formula~\cite[Proposition 2.7]{last_penrose_2017} then yields that this equals
	\be \label{eq:campbell}
	\sum_{k=2}^{n+1} \int_{[0,t]\cap[T_k,T_{k+1})}\Bigg(\frac{1}{R_n^{\ominus}}\sum_{u \in \cO_n^\square(T_k)}\sum_{\substack{v \in \cO_n^\ominus(T_k)\cup \cO_n^\oplus(T_k)\\ v\neq u}}\!\!\!\!\!\!\!\!\!\!\!\!\!\!r^\square\|u-v\|^{-\alpha_\square} \Bigg)f^\square_t(z)\, \dd z.
	\ee 
	Further, if we consider the quantity
	\be 
	\frac{1}{R_n^{\ominus}}\sum_{u \in \cO_n^\square(z)}\sum_{\substack{v \in \cO_n^\ominus(z)\cup \cO_n^\oplus(z)\\ v\neq u}}r^\square\|u-v\|^{-\alpha_\square} 
	\ee 
	as a function of $z$, then, by definition of the stopping times $T_k$, this function is a constant function on the interval $[T_k,T_{k+1})$. As a result, also recalling that $f^\square_t(z):=\e^{s^\square(t-z)}$, we can write~\eqref{eq:campbell} as
	\be \ba 
	\sum_{k=2}^{n+1} {}&\int_{[0,t]\cap[T_k,T_{k+1})}\Bigg(\frac{1}{R_n^{\ominus}}\sum_{u \in \cO_n^{\square}(z)}\sum_{\substack{v \in \cO_n^{\ominus}(z)\cup \cO_n^{\oplus}(z)\\ v\neq u}}\!\!\!\!\!\!\!\!\!\!\!\!\!\!r^\square\|u-v\|^{-\alpha_\square} \Bigg)\e^{s^\square(t-z)}\, \dd z \\
	&=\int_0^t \Bigg(\frac{1}{R_n^{\ominus}}\sum_{u \in \cO_n^{\square}(z)}\sum_{\substack{v \in \cO_n^{\ominus}(z)\cup \cO_n^{\oplus}(z)\\ v\neq u}}\!\!\!\!\!\!\!\!\!\!\!\!\!\!r^\square\|u-v\|^{-\alpha_\square} \Bigg)\e^{s^\square(t-z)}\, \dd z,
	\ea \ee 
	as desired.
\end{proof}

\noindent We are now ready to prove Proposition~\ref{prop:gen_size_artificial_CTBP_tail_bound}.

\begin{proof}[Proof of Proposition~\ref{prop:gen_size_artificial_CTBP_tail_bound}]
	We combine Lemma~\ref{lemma:Cexp} with Lemma~\ref{lemma:rateest}. {First, we apply Lemma~\ref{lemma:rateest}, with $A=\cO_n^{\square}(t)$ and $B=\cO_n^\ominus(z)\cup\cO_n^\oplus(z)$, to bound, for $\square\in\{\ominus, \oplus\}$,
		\be 
		\sum_{u \in \cO_n^{\square}(z)}\sum_{\substack{v \in \cO_n^{\ominus}(z)\cup \cO_n^{\oplus}(z)\\ v\neq u}}\!\!\!\!\!\!\!\!\!\!\!\!\!\!r^\square\|u-v\|^{-\alpha_\square}\leq |\cO_n^{\square}(z)|R^{\square}_{|\cO_n^\ominus(z)|+|\cO_n^\oplus(z)|}.
		\ee 
		We then combine this upper bound with Lemma~\ref{lemma:Cexp} to obtain the upper bound}
	\be 
	\E{|\cD_n^\square(t)|\,\Big|\,\cF}\leq \int_0^t \frac{1}{R_n^{\ominus}}|\cO_n^\square(z)|R^{\square}_{|\cO_n^\ominus(z)|+|\cO_n^\oplus(z)|}\e^{s^\square(t-z)}\, \dd z.
	\ee 
	Next, for $A>0$ we define the event 
	\be\label{eq:Markov_AV_D}
	\cE(A):=\{\sup_{t>0}|\cO_n^\ominus(t)|\e^{-t}<A\}\cap \{\sup_{t>0}|\cO_n^\oplus(t)|\e^{-Z_nt}<A\}.
	\ee   
	First, it follows from Lemma~\ref{lem:size_CTBP} for any $\eps>0$ and all $A=A(\eps)$ sufficiently large, that $\P{\cE(A)}\geq 1-\eps/2$, {since  $|\cO_n^\square(t)|\leq |\BP^{\square}(t)|$ by definition, for $\square\in\{\ominus, \oplus\}$.} Moreover, since $\cE(A)$ is measurable with respect to the $\sigma$-algebra $\cF$, we obtain using Markov's inequality, with $\theta>0$ and $A$ large, 
	\be \ba 
	\P{|\cD_n^\square(t)|\geq \theta}&\leq \frac1\theta \E{\cD_n^\square(t)\indicwo{\cE(A)}}+\P{\cE(A)^c}\\ 
	&\leq \frac1\theta\E{\indicwo{\cE(A)}\int_0^t \frac{1}{R_n^{\ominus}}|\cO_n^\square(z)| R^{\square}_{|\cO_n^\ominus(z)|+|\cO_n^\oplus(z)|}\e^{s^\square(t-z)}\, \dd z}+\eps/2.
	\ea \ee 
	Using the upper bounds in the event $\cE(A)$ and then bounding the indicator from above by one, yields the upper bound
	\be \ba \label{eq:Ccases}
	\P{|\cD_n^\square(t)|\geq \theta}&\leq \frac1\theta\int_0^t \frac{A}{R_n^{\ominus}}R^{\square}_{\floor{A(\e^z+\e^{Z_nz})}}\e^{s^\square t}\, \dd z+\eps/2\\
    &=\frac1\theta\int_0^t \frac{A s^\square}{R_n^{\square}}R^{\square}_{\floor{A(\e^z+\e^{Z_nz})}}\e^{s^\square t}\, \dd z+\eps/2 ,
	\ea \ee 
    where we recall that $s^\square=1$ if $\square=\ominus$ and $s^\square=Z_n=R_n^\oplus/R_n^\ominus$ if $\square=\oplus$. We   then bound 
	\be \label{eq:Rbounds}
	\underline c^\ominus\leq \frac{R_j^{\ominus}}{j^{1-\alpha_\ominus/d}}\leq \overline c^\ominus, \quad\text{and}\quad \underline c^\oplus\leq \frac{R_j^{\oplus}}{\lambda j^{1-\alpha_\oplus/d}}\leq \overline c^\oplus, 
	\ee 
	for all $j\in \N$ and for constants $0<\underline c^{\square}<\overline c^{\square}<\infty$ and for $\square\in\{\ominus,\oplus\}$, as follows from Lemma~\ref{lemma:ratesupinf}. We then use this in~\eqref{eq:Ccases} to arrive at  
	\be 
	\frac1\theta\int_0^t \frac{A^{2-\alpha_\square/d}\overline c^{\square} s^\square}{\underline c^\square}n^{-(1-\alpha_{\square}/d)} (\e^z+\e^{Z_nz})^{1-\alpha_{\square}/d}\e^{s^\square t}\, \dd z+\eps/2.
	\ee 
	Since we assume without loss of generality that $Z_n\geq 1$ (see Remark~\ref{rem:wlog}) and  by recalling that $\delta\in(0,\inf_{n\in\N}\min\{1-\alpha_\ominus/d,1-\alpha_\oplus/d\})$, we can further bound this from above, for some constants  $C_1=C_1(A)>0$, by 
	\be \ba 
    \frac1\theta\int_0^t C_1s^\square n^{-(1-\alpha_{\square}/d)} \e^{(1-\alpha_{\square}/d)Z_nz+s^\square t}\, \dd z+\eps/2\leq{}& \frac{C_1 s^\square}{\theta Z_n \delta} n^{-(1-\alpha_\square/d)}\e^{(s^\square+(1-\alpha_\square/d)Z_n)t}\\
    &+\eps/2.
	\ea \ee 
	Again using that $Z_n\geq 1$, so that $s^\square/Z_n\leq 1$, this leads to
	\be \label{eq:finC1}
	\P{|\cD_n^\square(t)|\geq \theta}\leq\frac{C_1}{\theta \delta } n^{-(1-\alpha_{\square}/d)}\e^{(s^\square+(1-\alpha_{\square}/d)Z_n)t}+\eps/2.
	\ee 
    We  apply this general upper bound to derive the desired results. First, we prove~\eqref{eq:Clin} for $\square=\ominus$. We set $t=(\log n-m)/Z_n$, and $\theta=n\e^{-m(1+\delta)}$, with $m$ a large constant to be determined. We also note that $s^\ominus=1$ and observe that we can assume without loss of generality that $\e^n>m$, since otherwise $t\leq 0$ and hence $|\cD_n^\ominus(t)|=0$ deterministically, so that the bound is trivial. Applying~\eqref{eq:finC1} yields 
	\be \ba 
	\mathbb P\bigg({}&\Big|\cD_n^\ominus\Big(\frac{\log n-m}{Z_n}\Big)\Big|\geq n\e^{-m(1+\delta)}\bigg)\\ 
	&\leq \frac{C_1}{\delta} \e^{m(1+\delta)} n^{-(2-\alpha_\ominus/d)}\e^{(1+(1-\alpha_\ominus/d)Z_n)(\log n-m)/Z_n} +\eps/2\\
	&= \frac{C_1}{\delta}\e^{m(1+\delta)-m(1-\alpha_\ominus/d)} n^{-1} \e^{(\log (n)-m)/Z_n}+\eps/2.
	\ea\ee 
	Since $Z_n\geq 1$ for all $n$ by assumption and $\e^n>m$, we can bound $n^{-1}\exp((\log (n)-m)/Z_n)\leq \e^{-m}$. We thus obtain the upper bound
    \be 
    \mathbb P\bigg(\Big|\cD_n^\ominus\Big(\frac{\log n-m}{Z_n}\Big)\Big|\geq n\e^{-m(1+\delta)}\bigg)\leq \frac{C_1}{\delta}\e^{m(\delta-(1-\alpha_\ominus/d))}+\eps/2.
    \ee 
    By the choice of $\delta$ we have that $\delta -(1-\alpha_\ominus(n)/d)$ is bounded away from zero uniformly in $n$. Hence, we can take $M=M(\delta,\eps)\in\N$ so large that for all $m\geq M$ the first term is at most $\eps/2$ to arrive at the desired upper bound.
    
    We then prove~\eqref{eq:Clin} for $\square=\oplus$, which follows a similar approach. We use~\eqref{eq:finC1} with $t=(\log n-m)/Z_n$ and $\theta=n\e^{-m(1+\delta)}$, and note that $s^\oplus=Z_n$. With similar calculations we obtain
	\be \ba 
	\mathbb P\bigg(\Big|\cD_n^\oplus\Big(\frac{\log n-m}{Z_n}\Big)\Big|\geq n\e^{-m(1+\delta)}\bigg)
	&\leq \frac{C_1}{\delta} \e^{m(1+\delta)} n^{-(2-\alpha_\oplus/d)}\e^{(2-\alpha_\oplus/d)(\log n-m)}+\eps/2\\
	&=\frac{C_1}{\delta} \e^{m(\delta -(1-\alpha_\oplus/d))} +\eps/2.
	\ea \ee 
	Again, by the choice of $\delta$ we can thus let $M=M(\delta,\eps)$ be so large that the first term is at most $\eps/2$ for all $m\geq M$ to arrive at the desired upper bound.
    
    To finally prove~\eqref{eq:Csublin}, we use~\eqref{eq:finC1} with $\square=\ominus$, $t=(1-\delta_n)\log(n)/Z_n$ and $\theta =\zeta \e^t$, where $\zeta>0$ is fixed, and $(\delta_n)_{n\in\N}$ is a sequence such that $\limsup_{n\to\infty} \delta_n<1$ and $\delta_n\log n$ diverges. We then obtain 
	\be \ba 
	\mathbb P\bigg(\Big|\cD_n^\ominus\Big(\frac{(1-\delta_n)\log n}{Z_n}\Big)\Big|\geq \zeta \e^{(1-\delta_n)\log(n)/Z_n}\bigg)&\leq \frac{C_1}{ \zeta\delta } n^{-(1-\alpha_\ominus/d)}\e^{(1-\alpha_\ominus/d)(1-\delta_n)\log n}+\eps/2\\ 
	&= \frac{C_1}{\zeta\delta } \e^{-(1-\alpha_\ominus/d)\delta_n\log n}+\eps/2. 
	\ea \ee 
	As $1-\alpha_\ominus/d$ is bounded away from zero since $\sup_{n\in\N}\alpha_\ominus(n)<d$, it follows that the first term on the right-hand side converges to zero as $n$ tends to infinity, irrespective of the choice of $\eps$ and $\zeta$. Thus, we obtain 
	\be 
	\limsup_{n\to\infty} \mathbb P\bigg(\Big|\cD_n^\ominus\Big( \frac{(1-\delta_n)\log n}{Z_n}\Big)\Big|\geq \zeta \e^{(1-\delta_n)\log(n)/Z_n}\bigg)\leq \eps/2.
	\ee 
	However, as the left-hand side does not depend on $\eps$, which was arbitrary, we in fact obtain that the limit of the probability equals zero, for any $\zeta>0$. This proves~\eqref{eq:Csublin} and concludes the proof.
\end{proof}

\section{Absence of coexistence}\label{sec:ub}

In this section, we prove an upper bound on $|\cN_n^\ominus(T_{\mathrm{cov}})|$ when $c_n \to \infty$, which is implied by the following proposition. Recall that a sequence of real-valued random variables $(X_n)_{n\in\N}$ is \emph{tight} when for any $\eps>0$ there exists $K=K(\eps)$ such that $\P{|X_n|>K}<\eps$.

\begin{proposition}\label{prop:infectsizes}
	Fix $d\in\N$, let $\alpha_\ominus,\alpha_\oplus\in[0,d)$ such that $\sup_{n\in\N}\alpha_\square(n)<d$ for $\square\in\{\ominus, \oplus\}$, and recall $Z_n$ from~\eqref{eq:Zn}. Let $N_n^\ominus:=|\cN_n^\ominus(T_{\mathrm{cov}})|$ denote the final size of the $\ominus$ infection of an LRC process, started from two distinct initial vertices $v^\ominus,v^\oplus\in\T_n^d$. When $c_n$ diverges, $(N_n^\ominus n^{-1/Z_n})_{n\in\N}$ is a tight sequence of random variables. 
\end{proposition}

\noindent To prove Proposition~\ref{prop:infectsizes}, we divide the competition process in two phases. In phase $\mathrm I$, the $\ominus$ infection grows to a sub-linear size of the `correct order', and the $\oplus$ infection grows linearly large. In  phase $\mathrm{II}$, the $\ominus$ infection only increases its size up to a multiplicative constant, that is, it stays of the same order, whereas the $\oplus$ infection is able to infect all the remaining vertices. This intuitive two-phase approach is outlined more precisely in the following two lemmas. To this end, we define, for $A,m>0$ and $\xi\in(0,1)$, the following event:
\be\ba \label{eq:en}
\cE_n(A,m,\xi):={}& \bigg\{\Big|\cN_n^\ominus\Big(\frac{\log n-m}{Z_n}\Big)\Big|\leq A\exp\Big( \frac{\log n-m}{Z_n}\Big)\bigg\}\\ 
&\cap\bigg\{\Big|\cN_n^\oplus\Big(\frac{\log n-m}{Z_n}\Big)\Big|\geq \xi n \bigg\}.
\ea \ee 
We can then formulate the following result, regarding phase $\mathrm I$:

\begin{lemma}[Phase $\mathrm I$]\label{lemma:ubinfection}
	Fix $d\in\N$, let $\alpha_\ominus,\alpha_\oplus\in[0,d)$ such that $\sup_{n\in\N}\alpha_\square(n)<d$ for $\square\in\{\ominus, \oplus\}$, and recall $Z_n$ from~\eqref{eq:Zn}.  When $c_n$ diverges, for any $\eps>0$ and all fixed $A,m$ sufficiently large and $\xi=\xi(m)$ sufficiently small, we have for all $n$ large that
	\be
	\P{\cE_n(A,m,\xi)}\geq 1-\eps.
	\ee 
\end{lemma}

\noindent When $c_n$ diverges, it follows that $\log(n)/Z_n=\log n - \cO(c_n\wedge \log n)$, so that the upper bound on the size of the $\ominus$ infection in Proposition~\ref{lemma:ubinfection} is indeed sub-linear, whereas the $\oplus$ infection has a linear size by time $(\log n-m)/Z_n$. Additionally we have the following tightness result for Phase $\mathrm{II}$. 

\begin{lemma}[Phase $\mathrm{II}$]\label{lemma:phase2} Fix $A,m>0$ and $\xi\in(0,1)$ as in Lemma~\ref{lemma:ubinfection}. Conditionally on the event $\cE_n(A,m,\xi)$, $(N_n^\ominus n^{-1/Z_n})_{n\in\N}$ is a tight sequence of random variables.
\end{lemma}

\noindent Proposition~\ref{prop:infectsizes} then readily follows: 

\begin{proof}[Proof of Proposition~\ref{prop:infectsizes} subject to Lemmas~\ref{lemma:ubinfection} and~\ref{lemma:phase2}.]
	Fix $\eps>0$ and choose $A$ and $m$ sufficiently large and $\xi$ sufficiently small so that $\P{\cE_n(A,m,\xi)^c}\leq \eps/2$, using Lemma~\ref{lemma:ubinfection}. Then, choose $K>0$ large so that $\mathbb P\big(N_n^\ominus n^{-1/Z_n}\geq K\,\big|\, \cE_n(A,m,\xi)\big)\leq \eps/2$, using Lemma~\ref{lemma:phase2}. It then follows that 
	\be 
	\P{N_n^\ominus n^{-1/Z_n}\geq K}\leq \P{N_n^\ominus n^{-1/Z_n}\geq K\,\Big|\, \cE_n(A,m,\xi)}+\P{\cE_n(A,m,\xi)^c}\leq \eps, 
	\ee 
	as desired.
\end{proof}

\noindent It thus remains to prove the two lemmas. The bulk of the work for Lemma~\ref{lemma:ubinfection} is already carried out in Section~\ref{sec:defect}, in Proposition~\ref{prop:gen_size_artificial_CTBP_tail_bound} in particular, so that we can provide the proof straight away with relative ease. Lemma~\ref{lemma:phase2} requires a bit more work, and is based on a P\'olya urn argument, which we take care of in Section~\ref{sec:polyaapprox}.

\begin{proof}[Proof of Lemma~\ref{lemma:ubinfection}]
	We set $t_n:=(\log n-m)/Z_n$ for ease of writing. We prove that the event $\cE_n$, as in~\eqref{eq:en}, holds with probability at least $1-\eps$ in two steps. First, as  $\cN_n^\ominus(t)\overset \dd= \cO^\ominus_n(t)\subseteq \cX_n^\ominus(t)$ for all $t\geq0$ by the coupling in Proposition~\ref{prop:coup_LRC_CTBRW}. As $|\cX_n^\ominus(t)|=|\BP^\ominus(t)|$, with $\BP^\ominus$ a rate one Yule process, we directly obtain for any $\eps>0$ and $A$ sufficiently large, that 
	\be \label{eq:N1bound}
	\P{|\cO_n^\ominus(t_n)|>A\e^{t_n}}\leq \P{|\BP^\ominus(t_n)|>A\e^{t_n}}\leq \P{\sup_{t\geq 0}|\BP^\ominus(t)|\e^{-t}>A}< \eps/2, 
	\ee 
	where the final step follows from Lemma~\ref{lem:size_CTBP}. Then, recalling $\cD_n^\oplus$ from~\eqref{eq:Dn}, we observe that 
	\be 
	|\cO_n^\oplus(t_n)|=|\cX_n^\oplus(t_n)|-|\cD_n^\oplus(t_n)|.
	\ee
	We recall that $\BP^\oplus$ is a Yule process with rate $Z_n$ and that $|\cX_n^\oplus(t)|=|\BP^\oplus(t)|$ for all $t\geq0$. As a result, for $\delta\in(0,\inf_{n\in\N}\min\{1-\alpha_\ominus/d,1-\alpha_\oplus/d\})$, the event 
	\be 
	\cG_n:=\{|\BP^\oplus(t_n)|\geq m^{-1}\e^{-m}n, \cD_n^\oplus(t_n)< n\e^{-m(1+\delta)}\}
	\ee 
	implies that, with $\xi<(m^{-1}-\e^{-\delta m})\e^{-m}$ (where the right-hand side is positive for $m$ large),
	\be 
	|\cO_n^\oplus(t_n)|\geq (m^{-1}-\e^{-\delta m})\e^{-m}n>\xi n. 
	\ee 
	The event $\cG_n$ holds with probability at least  $1-\eps/2$ for all $m$ and $n$ large, by Lemma~\ref{lem:size_CTBP} (using a similar proof as in~\eqref{eq:N1bound}, now using the infimum rather than the supremum) and Proposition~\ref{prop:gen_size_artificial_CTBP_tail_bound} (by the choice of $\delta$). Combining this with the bound in~\eqref{eq:N1bound} and Proposition~\ref{prop:coup_LRC_CTBRW}, we thus arrive at the desired conclusion.
\end{proof}

\subsection{Phase II: Proof of Lemma~\ref{lemma:phase2}}\label{sec:polyaapprox}

The event $\cE_n(A,m,\xi)$ implies that, by time $t_n:=(\log n-m)/Z_n$, the $\oplus$ infection has infected at least $k_n^\oplus:=\xi n$ many vertices and the $\ominus$ infection has infected at most $k_n^\ominus:=A\exp((\log n-m)/Z_n)$ many vertices.

In Phase II of the process, we aim to show that the $\ominus$ infection only reaches a small portion of the remaining number of uninfected vertices. Heuristically this should be the case, independently of \emph{which} vertices are exactly infected so far, due to the overwhelming advantage that the $\oplus$ infection already has. To make this heuristic argument precise, we use a discrete-time description of the infection process. Recall the stopping times $T_k$ from~\eqref{eq:Tkdef}. By Proposition~\ref{prop:coup_LRC_CTBRW} we can replace $\cO^\square_n(t)$ by $\cN_n^\square(t)$ in the definition of $T_k$. We note that $T_{\mathrm{cov}}=T_n$ and we set $T_2=0$. The discrete time process $(\cN_n^\ominus(T_k),\cN_n^\oplus(T_k))_{k=2}^n$ is then defined as follows. We set $\cN_n^\ominus(T_2)=\{v^\ominus\}$ and $\cN_n^\oplus(T_2)=\{v^\oplus\}$. For any $k\in\{2,\ldots, n-1\}$, conditionally on $(\cN_n^\ominus(T_k),\cN_n^\oplus(T_k))$, the pair $(\cN_n^\ominus(T_{k+1}),\cN_n^\oplus(T_{k+1}))$ is obtained by first sampling a vertex $v  \in (\cN_n^\ominus(T_k)\cup \cN_n^\oplus(T_k))^c$ with probability 
\begin{align*}
	\frac{\sum_{x \in \cN_n^\ominus(T_k)}\|x-v\|^{-\alpha_\ominus}+\sum_{y \in \cN_n^\oplus(T_k)}\lambda\|y-v\|^{-\alpha_\oplus}}{\sum_{z \in (\cN_n^\ominus(T_k)\cup \cN_n^\oplus(T_k))^c}\left(\sum_{x \in \cN_n^\ominus(T_k)}\|x-z\|^{-\alpha_\ominus}+\sum_{y \in \cN_n^\oplus(T_k)}\lambda\|y-z\|^{-\alpha_\oplus}\right)}.
\end{align*} 
Conditionally on both $v$ and $(\cN_n^\ominus(T_k), \cN_n^\oplus(T_k))$, with probability
\begin{align*}
	p_k:=\frac{\sum_{x \in \cN_n^\ominus(T_k)}\|x-v\|^{-\alpha_\ominus}}{\sum_{x \in \cN_n^\ominus(T_k)}\|x-v\|^{-\alpha_\ominus}+\sum_{y \in \cN_n^\oplus(T_k)}\lambda\|y-v\|^{-\alpha_\oplus}},
\end{align*}
we set $(\cN_n^\ominus(T_{k+1}),\cN_n^\oplus(T_{k+1}))=(\cN_n^\ominus(T_k)\cup\{v\},\cN_n^\oplus(T_k))$, and with probability $1-p_k$ we set $(\cN_n^\ominus(T_{k+1}),\cN_n^\oplus(T_{k+1}))=(\cN_n^\ominus(T_k),\cN_n^\oplus(T_k)\cup\{v\})$.

By the memoryless property of the exponential distribution and the fact that the minimum of exponential random variables is again exponentially distributed, the continuous-time process, sampled at the stopping times $T_k$, is indeed equal in distribution to this discrete-time process.

From here on out, we work conditionally on the event $\cE_n(A,m,\xi)$, and omit its arguments for ease of writing. Before proving Lemma~\ref{lemma:phase2}, let us introduce some notation. We set $\ell_n:=n-(|\cN_n^\ominus(t_n)|+|\cN_n^\oplus(t_n)|)$ and let $(A_j)_{j=0}^{\ell_n},(B_j)_{j=0}^{\ell_n}$ denote the sets of vertices that belong to the $\ominus$ and $\oplus$ infections when exactly $|\cN_n^\ominus(t_n)|+|\cN_n^\oplus(t_n)|+j$ many vertices have been infected, respectively. In particular, $A_0=\cN_n^\ominus(t_n)$, $B_0=\cN_n^\oplus(t_n)$, and $A_{\ell_n}=\cN_n^\ominus(T_{\mathrm{cov}})$, $B_{\ell_n}=\cN_n^\oplus(T_{\mathrm{cov}})$. The evolution of these sets can be described using the discrete analogue of the competition process introduced above.

\begin{proof}[Proof of Lemma~\ref{lemma:phase2}]
	For $j\in[\ell_n]$, let $\cF_j$ denote the $\sigma$-algebra generated by the sets $(A_i)_{i=0}^j $ and $(B_i)_{i=0}^j$. We note that the event $\cE_n$ is $\cF_j$ measureable for any $j\in[\ell_n]$.  We write 
	\be\ba 
	\mathbb E\big[{}&|A_{j+1}|\,\big|\, \cF_j \big]\\
    &=|A_j|+\frac{\sum_{v\in A_j}\sum_{u\in(A_j\cup B_j)^c}\!\|u-v\|^{-\alpha_\ominus}}{\sum_{v\in A_j}\sum_{u\in (A_j\cup B_j)^c}\!\|u-v\|^{-\alpha_\ominus}+\sum_{v\in B_j}\sum_{u\in (A_j\cup B_j)^c}\lambda \|u-v\|^{-\alpha_\oplus}}.
	\ea \ee 
	We bound the denominator from below by omitting the first sum and, using Lemma~\ref{lemma:rateest}, we bound the numerator from above by $|A_j|R^\ominus_{|(A_j \cup B_j)^c|}$ and the remaining term in the denominator from below by $|(A_j\cup B_j)^c|(R_n^\oplus-R_{|B_j^c|}^\oplus)$. Overall, we obtain the upper bound
	\be 
	\E{|A_{j+1}|\,|\, \mathcal F_j}\leq |A_j|\Big(1+ \frac{R_{|(A_j\cup B_j)^c|}^\ominus}{|(A_j\cup B_j)^c|(R^\oplus_n-R^\oplus_{|B_j^c|})}\Big).
	\ee 
    We then use the following estimate. Since, conditionally on the event $\cE_n$, $|B_0^c|=|\cN_n^\oplus(t_n)^c|\leq (1-\xi) n$ and as $|B_j^c|$ is non-increasing in $j$, it follows that, uniformly in $j$, 
	\be \label{eq:Rlb}
	R_n^\oplus-R_{|B_j^c|}^\oplus\geq R_n^\oplus-R_{|B_0^c|}^\oplus\geq \gamma R_n^\oplus, 
	\ee 
	almost surely, where $\gamma\in(0,1)$ is constant and we use Lemma~\ref{lemma:ratesupinf} in the final step. By combining this deterministic lower bound with the fact that $|(A_j\cup B_j)^c|=n-(|\cN_n^\ominus(t_n)|+|\cN_n^\oplus(t_n)|)-j=\ell_n-j$, we obtain 
    \be 
    \E{\mathbbm 1_{\cE_n}|A_{j+1}|}=\E{\mathbbm 1_{\cE_n}\E{|A_{j+1}|\,|\, \cF_j}}\leq \E{\mathbbm 1_{\cE_n}|A_j|\Big(1+\frac{1}{\gamma R_n^\oplus}\frac{R^\ominus_{\ell_n-j}}{\ell_n-j}\Big)}.
    \ee 
    Continuing this recursion yields
    \be 
    \E{\mathbbm 1_{\cE_n}N^\ominus_n}=\E{\mathbbm 1_{\cE_n}|A_{\ell_n}|}\leq \E{\mathbbm 1_{\cE_n}|A_0|\prod_{j=0}^{\ell_n-1} \Big(1+\frac{1}{\gamma R^\oplus_n}\frac{R^\ominus_{\ell_n-j}}{\ell_n-j}\Big)}.
    \ee 
    Since we have $\ell_n\leq n-k_n^\oplus\leq (1-\xi) n$ and $|A_0|\leq k_n^\ominus$ on the event $\cE_n$, we arrive at 
    \be \ba 
    \E{\mathbbm 1_{\cE_n} N_n^\ominus}\leq k_n^\ominus\exp\bigg(\frac{1}{\gamma R^\oplus_n}\sum_{j=1}^{(1-\xi) n} \frac{R^\ominus_j}{j}\bigg).
    \ea\ee 
    We then use the inequalities in Lemma~\ref{lemma:ratesupinf} to further bound this from above, for some constant $K_1>0$, by 
	\be
	k_n^\ominus\exp \bigg( \frac{K_1}{Z_nn^{1-\alpha_\ominus/d}}\sum_{j=1}^{(1-\xi) n} j^{-\alpha_\ominus/d}\bigg),
	\ee 
	By bounding the sum from above by an integral, using that $Z_n\geq 1$ and using Markov's inequality, we obtain, for some constants $K_2>0$ and $C>0$ the upper bound 
	\begin{align*}
		\P{N_n^\ominus\geq C k_n^\ominus\,\Big|\, \cE_n }\leq \frac{\E{\mathbbm 1_{\cE_n}N_n^\ominus }}{Ck_n^\ominus\P{\cE_n}}\leq C^{-1}\e^{K_2/Z_n}\leq \frac{1}{C(1-\eps)}\e^{K_2},
	\end{align*}
    where we bound $\P{\cE_n}\geq 1-\eps$ in the final step. We can make the right-hand side arbitrarily small by choosing $C$ large, which implies tightness of $N_n^\ominus/k_n^\ominus$, conditionally on $\cE_n $. Since $k_n^\ominus=A\exp((\log n-m)/Z_n)\leq A n^{1/Z_n}$, it follows that $N_n^\ominus n^{-1/Z_n}$ is tight, conditionally on $\cE_n $, as well, so that the desired result follows.
\end{proof}

\section{Coexistence and asymptotic behaviour of $N_n^\ominus$} \label{sec:nocoex}

In this section, we prove Theorems~\ref{thrm:coex} and~\ref{thrm:nocoex}. We remark that Proposition~\ref{prop:infectsizes} in Section~\ref{sec:ub} already implies that the size $N_n^\ominus$ of the $\ominus$ infection is $o(n)$ when $c_n$ diverges, which partially proves Theorem~\ref{thrm:nocoex}. We verify this in the proof of Theorem~\ref{thrm:nocoex}, in Section~\ref{sec:lb}. To obtain the asymptotic results of Theorems~\ref{thrm:nocoex}~\ref{case:1} and~\ref{case:fin}, we improve on Proposition~\ref{lemma:ubinfection} in Sections~\ref{sec:winnerone} and~\ref{sec:winnerfinite}, respectively, by using the approximation of the competition process to two branching random walks, as described in Section~\ref{sec:coupling}. Though we already used this to prove Proposition~\ref{prop:infectsizes}, we use it in a more refined way here to prove that $N_n^\ominus$ converges in distribution in these two cases. Next, we prove a lower bound for $N_n^\ominus$ for Theorems~\ref{thrm:nocoex}~\ref{case:liminf} and~\ref{case:div} in Section~\ref{sec:lb}, that matches the upper bound in Proposition~\ref{prop:infectsizes}. This section also contains the proof of Theorem~\ref{thrm:nocoex}. Finally, in Section~\ref{sec:coex}, we prove Theorem~\ref{thrm:coex}.

\subsection{The winner takes it all but one}\label{sec:winnerone}

To be able to show that the $\oplus$ infection reaches all but a finite number of vertices in this section and in Section~\ref{sec:winnerfinite}, we leverage a different description of the infection processes. Rather than thinking of vertices spreading an infection at a certain rate to other, uninfected, vertices, we assign to each edge $e$ between two vertices in $\T_n^d$ a transmission time $T^\square_e$. This denotes the time required for the  $\square$ infection to traverse the edge $e$ (given one of the incident vertices has the $\square$ infection and the other is uninfected). Here, $T^\square_e$ is an exponential random variable with rate $(R_n^\ominus)^{-1}\|e\|^{-\alpha_\ominus}$ when $\square=\ominus$ and rate $(R_n^\ominus)^{-1}\lambda \|e\|^{-\alpha_\oplus}$ when $\square=\oplus$. Clearly, this is an equivalent viewpoint as before, but we can now let $X_{u,v}^\square$ denote the edge-weighted distance between the vertices $u$ and $v$, formally defined as 
\be \label{eq:dist}
X_{u,v}^\square:=\inf_{\pi \in \Gamma_{u,v}}\sum_{e\in\pi} T_e^\square,
\ee 
where $\Gamma_{u,v}$ is the set that consists of all self-avoiding paths from $u$ to $v$ in $\T_n^d$. Studying the behaviour of $X^\square_{u,v}$ allows us to make conclusions regarding the spread of the $\square$ infection. 

The $X^\square_{u,v}$ are the object of study in an article of the first author and Van der Hofstad~\cite{HofLod23}. We generalise this to the following extent. Let $V\subset \T_n^d$ such that $u,v\in \T_n^d\setminus V$, and let $\Gamma_{u,v}^{(V)}$ be the set of self-avoiding paths from $u$ to $v$ in $\T_n^d\setminus V$. Then, define 
\be \label{eq:distV}
X^{\square,(V)}_{u,v}:=\inf_{\pi \in \Gamma^{(V)}_{u,v}}\sum_{e\in \pi}T^\square_e, 
\ee 
as the length of the shortest edge-weighted path from $u$ to $v$ in $\T_n^d\setminus V$.

We have the following result for the scaling of $X^{\oplus,(V)}_{u,v}$ for arbitrary sets $V$ that are `not too large'.

\begin{proposition}\label{prop:uniftime}
	Let $V\subset \T_n^d$ such that $|V|\leq\log n$, take $v\in\T_n^d\backslash V$, and let $U$ be a vertex selected uniformly at random from $\T_n^d\backslash (V\cup\{v\})$, independently of everything else. Run a single $\oplus$ infection  on $\T_n^d\backslash V$, started from $v$, with the associated distance metric as in~\eqref{eq:distV}. Finally let $t_n:=(1+\delta_n)\log(n)/Z_n$, where $\delta_n>0$ is such that $\delta_n\log n\to \infty$. Then, there exists $\kappa>0$ such that, uniformly over all sets $V$ of size at most $\log n$ and all $v\in \T_n^d\setminus V$,
	\be \label{eq:unifbound}
	\P{X_{v,U}^{\oplus,(V)}\geq t_n}\leq \e^{-\kappa \delta_n \log n}.
	\ee        
    Moreover, with $\delta_n\equiv \delta>0$ fixed, uniformly over all sets $V$ of size at most $\log n$ and all $v\in\T_n^d\setminus V$, 
    \be\label{eq:maxbound}
    \lim_{n\to\infty}\mathbb P\bigg(\max_{\substack{w\in\T_n^d\setminus V\\ w\neq v}}X_{v,w}^{\oplus,(V)}\geq 2t_n\bigg)=0.
    \ee 
\end{proposition} 

\begin{remark}
    The condition that the size of the set $V$ is at most $\log n$ can be weakened to $|V|\leq n^{1-\eps}$ for any $\eps>0$. However, as we only need an upper bound of $\log n$, we stated the results in this form.
\end{remark}

\begin{proof}
	We can bound the probability in~\eqref{eq:unifbound} by replacing $U$ with a supremum over all vertices in $\T_n^d\setminus V$ that are not $v$. That is, 
	\be\label{eq:unifapply}
	\P{X_{v,U}^{\oplus,(V)}\geq t_n}\leq \sup_{\substack{w\in \T_n^d\backslash V\\w \neq v}}\P{X_{v,w}^{\oplus,(V)}\geq t_n}.
	\ee 
    Similarly, we can bound the probability in~\eqref{eq:maxbound} from above by
    \be \label{eq:maxapply}
    \mathbb P\bigg(\max_{\substack{w\in\T_n^d\setminus V\\ w\neq v}}X_{v,w}^{\oplus,(V)}\geq 2t_n\bigg)\leq n\sup_{\substack{w\in \T_n^d\backslash V\\w \neq v}}\P{X_{v,w}^{\oplus,(V)}\geq 2t_n}.
    \ee 
    We bound $\P{X_{v,w}^{\oplus,(V)}\geq t_n}$ from above by a term that is independent of the choice of $V$ and the choice of $w$. First, instead of viewing $X_{v,w}^{(V)}$ as the time it takes for an $\oplus$ infection, started from $v$, to reach $w$ on $\T_n^d\setminus V$, we view $X_{v,w}^{(V)}$ as the time it takes for two $\oplus$ infection processes, started from $v$ and $w$, respectively, to meet. That is, until the set of vertices the infection started at $v$ has discovered is no longer disjoint from the set of vertices discovered by the infection started at $w$. The two infection processes are independent only up until the time that they are no longer disjoint. As a result, we can dominate $X_{v,w}^{\oplus,(V)}$ by the time it takes two $\oplus$ infection processes on $\T_n^d\b V$, started at $v$ and $w$, to reach $a_n:=\lfloor c\sqrt n\rfloor$ many individuals, plus the shortest transmission time $T^\oplus_{(x,y)}$, among vertices $x$ and $y$ in the processes started at $v$ and $w$, respectively. Here, $c>1$ is a constant to be determined, and we stress that $T^\oplus_{(x,y)}$ is the time it takes $x$ to infect $y$ (or vice versa) along a \emph{single} edge $\{x,y\}$. We let $\tau_j^{(v)}$ denote the time it takes for the infection started from $v$ to reach $j\in [n-|V|]$ many vertices, and let $B_v(\tau_j^{(v)})$ denote the set of vertices, reached by $v$ at time $\tau_j^{(v)}$. Equivalent notation is used for the infection started from vertex $w$. We then have the following upper bound valid for any $V$ and $w$,
	\be \label{eq:maxprob}
	\P{X_{v,w}^{\oplus,(V)}\geq t_n} \leq  \P{\tau_{a_n}^{(v)}+\tau_{a_n}^{(w)}+\min_{\substack{x\in B_w(\tau^{(w)}_{a_n}),y\in B_v(\tau^{(v)}_{a_n})}} T^\oplus_{(x,y)}\geq t_n}.
	\ee 
	Indeed, if the two sets  $B_v(\tau^{(v)}_{a_n})$ and $B_w(\tau^{(w)}_{a_n})$ are disjoint, then the sum of these three terms provides an upper bound for the edge-weighted distance of one possible path from $w$ to $v$, and hence is an upper bound for $X_{v,w}^{\oplus,(V)}$. If the two sets are not disjoint, then the minimum equals zero and the first two terms already are an upper bound for $X_{v,w}^{\oplus,(V)}$, since $v$ and $w$ can be connected within time $\tau_{a_n}^{(v)}+\tau_{a_n}^{(w)}$. We further split this probability into the two terms
    \be\label{eq:probsplit}
	\P{\tau_{a_n}^{(v)}+\tau_{a_n}^{(w)}\geq \Big(1+\frac{\delta_n}{2}\Big)\frac{\log n}{Z_n}}+\P{\min_{\substack{x\in B_w(\tau^{(w)}_{a_n}),y\in B_v(\tau^{(v)}_{a_n})}} T^\oplus_{(x,y)}\geq\frac{\delta_n}{2}\frac{\log n}{Z_n}}. 
	\ee 
    Let us start by bounding the second term from above. As a minimum of exponential random variables is an exponential random variable, whose rate equals the sum of all the individual rates, we can condition on $B_w(\tau^{(w)}_{a_n})$ and $B_v(\tau^{(v)}_{a_n})$ to obtain       
	\be\ba 
	\mathbb P\bigg({}&\min_{\substack{x\in B_w(\tau^{(w)}_{a_n}),y\in B_v(\tau^{(v)}_{a_n})}}\!\!\!\! T^\oplus_{(x,y)}\geq \frac{\delta_n\log n}{2Z_n}\bigg)\\ 
	&= \mathbb E\Bigg[\exp\bigg(\!-\frac{\delta_n}{2}\frac{\log n}{Z_n}\frac{1}{R_n^\ominus}\!\!\!\!\sum_{x\in B_w(\tau^{(w)}_{a_n})}\sum_{y\in B_v(\tau^{(v)}_{a_n})}\!\!\!\!\!\!\!\lambda \|x-y\|^{-\alpha_\oplus}\bigg)\ind{B_v(\tau^{(v)}_{a_n})\cap B_v(\tau^{(v)}_{a_n})=\varnothing}\Bigg].
	\ea \ee 
	We now bound $\|x-y\|^{-\alpha_\oplus}\geq  Cn^{-\alpha_\oplus/d}$ for some large constant $C>0$, since we can bound the distance between $x$ and $y$ from above by a constant times $n^{1/d}$ for any $x,y\in \T_n^d$. Further using that $Z_nR_n^\ominus=R_n^\oplus$, the fact that $R_n^\oplus\lambda^{-1}=\Theta(n^{1-\alpha_\oplus/d})$ by Lemma~\ref{lemma:ratesupinf} and that $a_n=\lfloor c\sqrt n\rfloor$, we obtain for some constants $C,C'>0$, the upper bound
	\be \label{eq:minedgebound}
	\exp\Big(-\frac{\delta_n}{2}\frac{\log n}{R^\oplus_n} \lambda a_n^2 Cn^{-\alpha_\oplus/d}\Big)\leq \exp\Big(-c^2C'\frac{\delta_n}{2}\log n\Big).
	\ee
    We then turn to the first term in~\eqref{eq:probsplit}. We observe that the rate $\Lambda_j^{(V)}$ at which the infection reaches the next vertex, given that the infection consists of $j\in[a_n-1]$ many vertices, satisfies the lower bound 
    \be 
    \Lambda_j^{(V)}\geq j(R_n^\ominus)^{-1}(R_n^\oplus-R_{|V|+(j-1)}^\oplus)\geq jZ_n(1-d_n),
    \ee 
    almost surely, by Lemma~\ref{lemma:rateest}, where $d_n:=R^\oplus_{3a_n}/R_n^\oplus$. As a result, with $(E_j^{(v)})_{j\in\N}$ a sequence of i.i.d.\ rate-one exponential random variables, we can almost surely bound 
	\be \label{eq:tauub}
	\tau_{a_n}^{(v)}=\sum_{j=1}^{a_n-1}\frac{E_j^{(v)}}{\Lambda_j^{(V)}}\leq \frac{1}{Z_n(1-d_n)}\sum_{j=1}^{a_n-1}\frac{E_j^{(v)}}{j}.
	\ee 
    Further, we stochastically dominate $\tau_{a_n}^{(w)}$ by $\wt\tau_{a_n}^{(w)}$, which denotes the time it takes a $\oplus$ infection process that starts from $w$ and is run on $\T_n^d\setminus(V\cup B_v(\tau_{a_n}^{(v)}))$ to reach $a_n$ vertices, conditionally on $\tau_{a_n}^{(v)}$ and $B_v(\tau_{a_n}^{(v)})$. Indeed, such a process has fewer vertices that it can spread to, compared to the process that starts from $w$ but can spread on all of $\T_n^d$. Hence, it takes a longer time to reach $a_n$ many vertices, so that $\tau_{a_n}^{(w)}\preceq \wt\tau_{a_n}^{(w)}$. In a similar manner, we let $\wt \Lambda_j^{(V)}$ denote the rate at which this infection reaches the next vertex, given that the infection consists of $j$ vertices. We now have the lower bound 
    \be 
    \wt \Lambda_j^{(V)}\geq j(R_n^\ominus)^{-1}(R_n^\oplus-R^\oplus_{|V|+a_n+(j-1)})\geq jZ_n(1- d_n),
    \ee 
    almost surely. With $(E^{(w)}_j)_{j\in\N}$ a sequence of i.i.d.\ rate-one exponential random variables (also independent of the $(E_j^{(v)})_j$), we thus have the almost sure upper bound 
    \be \label{eq:tautildeub}
    \wt \tau_{a_n}^{(w)}=\sum_{j=1}^{a_n-1}\frac{E_j^{(w)}}{\wt \Lambda_j^{(V)}}\leq \frac{1}{Z_n(1- d_n)}\sum_{j=1}^{a_n-1}\frac{E_j^{(w)}}{j}.
    \ee 
    Crucially, we observe that this upper bound is independent of $\tau_{a_n}^{(v)}$ and $B_v(\tau_{a_n}^{(v)})$. We combine this with~\eqref{eq:tauub} to obtain 
    \be\ba 
    \mathbb P\bigg(\tau_{a_n}^{(v)}+\tau_{a_n}^{(w)}\geq \Big(1+\frac{\delta_n}{2}\Big)\frac{\log n}{Z_n}\bigg)&\leq \P{\tau_{a_n}^{(v)}+\wt\tau_{a_n}^{(w)}\geq \Big(1+\frac{\delta_n}{2}\Big)\frac{\log n}{Z_n}}\\ 
    &\leq \P{\sum_{j=1}^{a_n-1}\frac{E_j^{(v)}+E_j^{(w)}}{j}\geq \Big(1+\frac{\delta_n}{2}\Big)(1-d_n)\log n}.
    \ea\ee 
    An exponential Markov inequality with $s\in(0,1)$ then yields 
    \be\ba 
    \mathbb P\bigg({}&\tau_{a_n}^{(v)}+\tau_{a_n}^{(w)}\geq \Big(1+\frac{\delta_n}{2}\Big)\frac{\log n}{Z_n}\bigg)\\
    &\leq \exp\Big(-s\Big(1 +\frac{\delta_n}{2}\Big)(1-d_n)\log n\Big)\mathbb E\bigg[\exp\bigg(\sum_{j=1}^{a_n-1} \frac sj E_j^{(v)}\bigg)\bigg]^2\\
	&\leq \exp\Big(-s\Big(1 +\frac{\delta_n}{2}\Big)(1-d_n)\log n\Big)\bigg(\prod_{j=1}^{a_n-1}\frac{j}{j-s}\bigg)^2.
    \ea\ee 
	We bound the product from above by 
	\be 
	\prod_{j=1}^{a_n-1}\frac{j}{j-s}\leq\exp\bigg(s\sum_{j=1}^{a_n-1}\frac{1}{j-s}\bigg)=\exp(s\log(a_n)+\cO(1))=\exp\Big(\frac s2\log n+\cO(1)\Big).
	\ee 
    Furthermore, we use~\eqref{eq:Ralphan} to obtain 
	\be 
	d_n=(1+o(1))\Big(\frac{3a_n}{n}\Big)^{1-\alpha_\oplus/d}\leq 3c n^{-(1-\alpha_\oplus/d)/2}=o(\delta_n), 
	\ee 
    where the last step holds since $\sup_{n\in\N}\alpha_\oplus(n)<d$ and $\delta_n\log n\to\infty$. 
	Combined, we thus arrive at 
	\be \ba \label{eq:caseibound}
	\P{\tau_{a_n}^{(v)}+\tau_{a_n}^{(w)}\geq \Big(1+\frac{\delta_n}{2}\Big)\frac{\log n}{Z_n}}&\leq \exp\Big(s \log n-s\Big(1 +\frac{\delta_n}{2}\Big)(1-d_n)\log n +\cO(1)\Big)\\
	&\leq \exp\Big(-s \delta_n\log (n)\Big[\frac12-\frac{d_n}{\delta_n}-\frac{d_n}{2}\Big]+\cO(1)\Big). 
	\ea \ee 
	Since $d_n=o(\delta_n)$ and $d_n=o(1)$, it follows that the terms within the square brackets converge to $1/2$. As we can choose $s$ arbitrarily close to $1$ and since $\delta_n \log n$ diverges with $n$, we obtain that the probability on the left-hand side is at most $\exp(-\kappa \delta_n\log n)$ for any constant $\kappa\in(0,1/2)$ when $n$ is large enough. Finally, we combine this with~\eqref{eq:minedgebound} and~\eqref{eq:probsplit} to obtain desired upper bound in~\eqref{eq:unifapply} and proves~\eqref{eq:unifbound}.

    A similar approach, replacing $t_n$ by $2t_n$ and using this in~\eqref{eq:maxapply} (where we recall that now $\delta_n\equiv \delta $ is fixed), yields
    \be \ba 
    \mathbb P\bigg(\max_{\substack{w\in\T_n^d\setminus V\\ w\neq v}}X_{v,w}^{\oplus,(V)}\geq 2t_n\bigg)\leq{}& \exp \Big(\big[(1+s)-s(2+\delta)(1-d_n)\big]\log n+\cO(1)\Big)\\
    & +\exp\Big(\big[1-c^2C'\delta\big]\log n\Big).
    \ea \ee 
    Since $d_n=o(1)$, we can choose $s=s(\delta)\in(0,1)$ close enough to one such that the first term in the upper bound tends to zero. For the second term we can choose $c$ large enough such that the terms in the brackets tend to zero, which yields~\eqref{eq:maxbound} and concludes the proof.
\end{proof}

With this result at hand, we prove Theorem \ref{thrm:nocoex}~\ref{case:1}.

\begin{proof}[Proof of Theorem~\ref{thrm:nocoex}~\ref{case:1}]
	Fix two distinct vertices $v^\ominus,v^\oplus\in\T_n^d$, set $V=\{v^\ominus\}$, and start the $\square$ infection from $v^\square$ for $\square\in\{\ominus, \oplus\}$. For $\beta>2$ fixed, we introduce the following two events:
	\be \label{eq:winevents}
	\cE_n^\ominus:=\Big\{\min_{\substack{u\in\T_n^d\\ u\neq v^\ominus,v^\oplus}}T_{v^\ominus,u}^\ominus>\frac{\beta\log n}{Z_n}\Big\}, \quad\text{and}\quad \cE_n^\oplus:=\Big\{\max_{\substack{u\in\T_n^d\\ u\neq v^\ominus,v^\oplus}}X_{v^\oplus,u}^{\oplus,(V)}<\frac{\beta\log n}{Z_n}\Big\}.
	\ee 
	In words, $\cE_n^\ominus$ states that it takes the $\ominus$ infection more than $\beta\log n/Z_n$ amount of time to traverse any edge incident to $v^\ominus$, and thus the $\ominus$ infection reaches its first vertex after time $\beta\log n/Z_n$. On the other hand, $\cE_n^\oplus$ states that the maximal weighted distance from $v^\oplus$ in $\T_n^d\setminus V=\T_n^d\setminus\{v^\ominus\}$ among all other vertices (not including $v^\ominus$) for the $\oplus$ infection is at most $\beta\log n/Z_n$. As a result, the event $\cE_n^\ominus\cap \cE_n^\oplus$ implies that the $\oplus$ infection reaches \emph{all} vertices in $\T_n^d\setminus\{v^\ominus\}$ before the $\ominus$ infection spreads to any vertex. Hence, it suffices to prove that the probability of $\cE_n^\ominus\cap\cE_n^\oplus$ converges to one to prove Case~\ref{case:1}. 
	
    We show that both events have a probability converging to one. For the second event, this follows directly from~\eqref{eq:maxbound} in Proposition~\ref{prop:uniftime}, when we set $V=\{v^\ominus\}$. For the event $\cE_n^\ominus$ we use that the minimal edge-weight among $(T_{v^\ominus,u}^\ominus)_{u\in\T_n^d\setminus\{v^\ominus,v^\oplus\}}$ is an exponential random variable with rate $1-\|v^\ominus-v^\oplus\|^{-\alpha_\ominus}/R_n^\ominus\leq 1$, to arrive at 
	\be \label{eq:En1}
	\P{\cE_n^\ominus}=\exp\Big(-\big(1-\|v^\ominus-v^\oplus\|^{-\alpha_\ominus}/R_n^\ominus\big)\frac{\beta\log n}{Z_n}\Big)\geq\exp\Big(-\frac{\beta\log n}{Z_n}\Big).
	\ee 
	If $c_n/(\log n)^2$ diverges or, equivalently, $Z_n/\log n$ diverges, it follows that $\mathbb P(\cE_n^\ominus)=1-o(1)$, which yields the desired result.
\end{proof}

\subsection{The winner takes all but finitely many}\label{sec:winnerfinite}

In Section~\ref{sec:winnerone}, we saw that the proof of Theorem \ref{thrm:nocoex}~\ref{case:1} is based on the fact that the $\oplus$ infection can spread to all vertices within time $\beta\log(n)/Z_n$ with high probability for any $\beta>2$, since $c_n\gg (\log n)^2$ and thus $Z_n\gg \log n$. In this section, we discuss Theorem \ref{thrm:nocoex}~\ref{case:fin}, where $c_n=(c+o(1))(\log n)^2$ and thus $Z_n=(c+o(1))\log n$. In this case, it is not sufficient to run the process up to time $2\log(n)/Z_n$. Instead, the aim is to run the competition process up to time $t_n=(1+\delta_n)\log(n)/Z_n$, for some small $\delta_n$, and show that by this time the $\oplus$ infection has reached all but a negligible proportion of the vertices, whilst the approximation of the slower $\ominus$ infection by the corresponding branching random walk $\cX_n^\ominus$ up to time $t_n$ is in fact \emph{exact} with high probability. A P\'olya urn argument, similar in spirit to that in the proof of Lemma~\ref{lemma:phase2}, is used to show that after time $t_n$ the $\oplus$ infection reaches \emph{all} remaining uninfected vertices before the $\ominus$ infection can make any more infections, so that the final size $N_n^\ominus$ is equal to the size of the branching random walk $\cX_n^\ominus$ at time $t_n\approx 1/c$. Combined with Lemma~\ref{lem:size_CTBP}, this yields the desired distributional limit.

\begin{proof}[Proof of Theorem~\ref{thrm:nocoex}~\ref{case:fin}]
	We let $(\delta_n)_{n\in\N}$ be a non-negative sequence such that $\delta_n\downarrow 0$ and $\delta_n\log n\to\infty$ as $n\to\infty$. It follows that $n^{1-\delta_n}=o(n)$. We also define $s_n:=(1-\delta_n)\log(n)/Z_n$ and $t_n:=(1+\delta_n)\log(n)/Z_n$. Finally, we define the events
	\be \ba \label{eq:events}
	\cE_{1,n} &:=\{\cO_n^\ominus(s_n)=\cX_n^\ominus(s_n)=\cX_n^\ominus(t_n), |\cX_n^\ominus(t_n)|\leq \log n\},\\
	\cE_{2,n}&:=\{|\cO_n^\oplus(s_n)|\leq n^{1-\delta_n/2},  |\cO_n^\oplus(t_n)|\geq n-n^{1-\xi \delta_n}\},\\
	\cE_{3,n}&:=\{\cO_n^\ominus(T_{\mathrm{cov}})=\cX_n^\ominus(t_n)\},
	\ea\ee 
	where $\xi>0$ is small and to be determined later, we recall that $T_{\mathrm{cov}}$ is the  time when all vertices of the torus are infected, that $(\cO_n^\ominus(t),\cO_n^\oplus(t))_{t\geq 0}$ is introduced in the coupling in Section~\ref{sec:coupling}, and Proposition~\ref{prop:coup_LRC_CTBRW} yields that $(\cO_n^\ominus(t),\cO_n^\oplus(t))_{t\geq 0}\overset\dd= (\cN_n^\ominus(t),\cN_n^\oplus(t))_{t\geq 0}$. The aim is to prove that $\P{\cE_{3,n}}$ tends to $1$ with $n$. Indeed, this is sufficient, since $t_n\to 1/c$ when $Z_n=(c+o(1))\log n$ implies that $|\cX_n^\ominus(t_n)|=|\BP^\ominus(t_n)|\toindis |\BP^\ominus(1/c)|$, where we recall $\BP^\ominus$ from the start of Section~\ref{sec:coupl}. As a result, we have  $N_n^\ominus=|\cN_n^\ominus(T_{\mathrm{cov}})|\toindis |\BP^\ominus(1/c)|$. The latter limiting random variable is equal in distribution to a geometric random variable with success probability $p=\e^{-1/c}$, as desired, which follows from Lemma~\ref{lem:size_CTBP}.
	
	Let us start by observing that the event $\cE_{1,n}$ implies that $\cN_n^\ominus(s_n)=\cN_n^\ominus(t_n)=\BP^\ominus(t_n)$, whereas the event $\cE_{2,n}$ implies that, by time $s_n$, there are still uninfected vertices, and by time $t_n$ the $\oplus$ infection has reached almost all vertices. Hence, on the event $\cE_{1,n}\cap\cE_{2,n}$, the event $\cE_{3,n}$ holds when all uninfected vertices after time $t_n$ (if any) are reached by the $\oplus$ infection. We start by proving that $\cE_{1,n}\cap \cE_{2,n}$ holds with high probability, and then turn to the event $\cE_{3,n}$. 
	
	To prove that $\P{\cE_{1,n}}$ tends to one, we recall that 
	\be 
	\cX_n^\ominus(s_n)=\cO_n^\ominus(s_n)\cup \cD_n^\ominus(s_n), 
	\ee 
	so that the first equality in the event $\cE_{1,n}$ holds when $| \cD^\ominus(s_n)|=0$. We use Proposition~\ref{prop:gen_size_artificial_CTBP_tail_bound} to show this holds with high probability. Further, the second equality in the event $\cE_{1,n}$ is implied if no particle in $\cX_n^\ominus(s_n)$ gives birth to a child in the time interval $(s_n,t_n]$. As a result, by a union bound and Markov's inequality, 
\begin{align}\label{eq:E1nbound}
    \P{\cE_{1,n}^c}\leq{}& \P{\cD_n^\ominus(s_n)\geq 1}+\P{\exists v\in \cX_n(s_n): v\text{ has a child in }(s_n,t_n]}\notag \\
	&+\P{|\cX_n^\ominus(t_n)|>\log n}\notag \\ 
	\leq{}& \P{\cD_n^\ominus(s_n)\geq 1}+\E{|\cX_n^\ominus(s_n)|}\big(1-\e^{-(t_n-s_n)}\big)+\frac{\E{|\cX_n^\ominus(t_n)|}}{\log n}.
\end{align}
	The probability on the right-hand side tends to zero by Proposition~\ref{prop:gen_size_artificial_CTBP_tail_bound}, which states that 
    \be 
	\cD_n^\ominus(s_n)\e^{-s_n}\toinp 0. 
	\ee 
	Since $Z_n=(c+o(1))\log n$ for some $c>0$, so that $s_n$ converges to $1/c$, it in fact holds that $\cD_n^\ominus(s_n)\toinp 0$. The two expected values equal $\e^{s_n}$ and $\e^{t_n}$, respectively, by Lemma~\ref{lem:size_CTBP}, as it follows from the definition of the branching random walk $\cX_n^\ominus$ that $|\cX_n^\ominus|=|\BP^\ominus|$. Both $\e^{s_n}$ and $\e^{t_n}$ converge to $\e^{1/c}$. Furthermore, as $t_n-s_n=2\delta_n\log(n)/Z_n=o(1)$, the entire right-hand side of~\eqref{eq:E1nbound} tends to zero with $n$.
	
	We then show that $\P{\cE_{2,n}^c}$ vanishes with $n$. First, we use that $|\BP^\oplus(s_n)|\e^{-Z_ns_n}$ converges almost surely (as $Z_ns_n$ diverges with $n$) by Lemma~\ref{lem:size_CTBP}. As a result, 
	\be \label{eq:E2prob1}
	|\cO_n^\oplus(s_n)|\leq |\cX_n^\oplus(s_n)|=|\BP^\oplus(s_n)|\leq \e^{Z_ns_n+\delta_n \log(n)/2}=n^{1-\delta_n/2},
	\ee 
	with high probability. Now, we prove that $|\cO^\oplus(t_n)|\geq n-n^{1-\xi\delta_n}$ holds with probability tending to one for any $\xi$ sufficiently small. Here, we use that $\cO_n^\oplus(t)\overset \dd= \cN_n^\oplus(t)$ for all $t\geq0$ and apply the following approach. 
    
    Let $(\mathbf{X}^{\ominus}_n(t))_{t\in[0,t_n]}$ denote the set of \emph{spatial locations} in $\T_n^d$ occupied by a \emph{single} branching random walk $\cX_n^\ominus$ during the interval $[0,t_n]$, i.e.,
    \begin{align*}
        \mathbf{X}^\ominus_n(t):=\{x_u:\exists u\in \mathcal{U}_\infty,m^\ominus_u\in \{\cA,\cO\},t\leq t_n\textrm{ with }(u,x_u,m^\ominus_u)\in \BP^\ominus(t)\}.
    \end{align*} 
    Note that $\mathbf{X}^\ominus_n(t)$ contains the spatial locations of both original and artificial particles. By definition, we have $\mathbf{X}^\ominus_n(t)\supseteq \cN_n^\ominus(t)$ for all $t\in[0,t_n]$.
    
    We then let $(\wt \cN_n^\oplus(t))_{t\in[0,t_n]}$ denote the set of infected vertices when we run a \emph{single} $\oplus$ infection on $\T_n^d\setminus  \mathbf{X}^\ominus_n(t_n)$ for $t_n$ amount of time. Again, since $\mathbf{X}^\ominus_n(t_n)\supseteq \cN_n^\ominus(t_n)\supseteq \cN_n^\ominus(t)$ for all $t\in[0,t_n]$, it follows that, conditionally on $\mathbf{X}^\ominus_n(t_n)$, we can couple $(\cN_n^\oplus(t))_{t\in[0,t_n]}$ and $(\wt \cN_n^\oplus(t))_{t\in[0,t_n]}$ such that $\wt \cN_n^\oplus(t)\subseteq \cN_n^\oplus(t)$ for all $t\in[0,t_n]$, almost surely: First, for each $x\in \T_n^d$, let $\xi_x^\oplus$ be a Poisson process with rate $Z_n$ and let $(X_{x,i}^\oplus)_{i\in\N}$ be a sequence of independent displacement random variables sampled from the displacement distribution $p_n^\oplus$. Here, the random variables $(\xi_x^\oplus,(X_{x,i}^\oplus)_{i\in\N})$ and $(\xi_y^\oplus,(X_{y,i}^\oplus)_{i\in\N})$ are also independent for $x\neq y$. We then construct $\cX_n^\oplus$ as follows. At time $t=0$, start with a single particle at location $v^\oplus\in\T_n^d$ and assign it the Poisson process $\xi_{v^\oplus}$ and random variables $(X_{v^\oplus,i}^\oplus)_{i\in\N}$. The points in $\xi_{v^\oplus}$ determine the birth-times of the children of the particle at $v^\oplus$ and the $i^{\mathrm{th}}$ child is born at location $v^\oplus+X_{v^\oplus,i}^\oplus\in \T_n^d$. Any particle that is the \emph{first} to be born at a location $x\in \T_n^d$ is assigned $\xi_x^\oplus$ and $(X_{x,i}^\oplus)_{i\in\N}$. Any particle that is born at a location where particles have already been born, is assigned an i.i.d.\ copy of $\xi_x^\oplus$ and $(X_{x,i}^\oplus)_{i\in\N}$. We construct $\cN_n^\oplus$ from $\cX_n^\oplus$ by discarding all artificial particles, that is, for each location $x\in\T_n^d$ we discard all particles that are not the firstborn particle at $x$ (where the initial particle at $v^\oplus$ is considered the firstborn particle at $v^\oplus$). 
    
    Second, conditionally on $\mathbf X_n^\ominus(t_n)$, we construct $\wt \cX_n^\oplus$ as follows. If $v^\oplus\in \mathbf X_n^\ominus(t_n)$, then $\wt \cX_n^\oplus(t)=\{v^\oplus\}$ for all $t\geq 0$. Otherwise, construct $\wt \cX_n^\oplus$ in the same manner as $\cX_n^\oplus$, but where we discard all particles born in $\mathbf X_n^\ominus(t_n)$ and also discard all their descendants. It is crucial here that the particle \emph{first} born at a location is assigned the \emph{same} Poisson process $\xi^\oplus_x$ and displacement random variables $(X_{x,i}^\oplus)_{i\in\N}$ as in the process $\cX_n^\oplus$. Note that the firstborn particle at a location $x$ need not be the same particle in $\cX_n^\oplus$ and $\wt \cX_n^\oplus$ (this is due to particles born at locations in $\mathbf X_n^\ominus(t_n)$ and their descendants being discarded). However, if the firstborn particle born at location $x$ is born in $\wt\cX_n^\oplus$ at time $\wt t_x$ and in $\cX_n^\oplus$ at time $t_x$, then we do always have $\wt t_x\geq t_x$ almost surely.  We construct $\wt \cN_n^\oplus$ from $\wt \cX_n^\oplus$ by discarding all artificial particles, that is, for each location $x\in\T_n^d$ discarding all particles that are not the firstborn particle at $x$.  
    
    From this coupling, we obtain that $\wt \cN_n^\oplus(t)\subseteq \cN_n^\oplus(t)$ for all $t\geq 0$ almost surely (since $\wt t_x\geq t_x$ for each $x\in\T_n^d$ almost surely).
    
    We now  define the event $\wt \cE_{1,n}:=\{|\mathbf{X}^\ominus_n(t_n)|\leq \log n\}$. We can then bound
    \be
	\mathbb P\Big(\big\{|\cN_n^\oplus(t_n)|<n-n^{1-\xi\delta_n}\big\} \cap\wt\cE_{1,n}\Big)\leq \mathbb E\Big[\indicwo{\wt \cE_{1,n}}\!\!\mathbb P\Big(|\wt \cN_n^\oplus(t_n)|<n-n^{1-\xi\delta_n}\,\Big|\,  \mathbf{X}^\ominus_n(t_n) \Big)\Big].  
\ee 
    We can bound the expected value from above by 
	\be 
	\sup_{\substack{D\subset \T_n^d\\ |D|\leq \log n}} \P{|\wt\cN_n^\oplus(t_n)|<n-n^{1-\xi\delta_n}\,\Big|\, \mathbf{X}^\ominus_n(t_n)=D}. 
	\ee  
	Then, using the distance metric in~\eqref{eq:distV}, we can write 
	\be \ba 
	\P{|\wt\cN_n^\oplus(t_n)|<n-n^{1-\xi\delta_n}\,\Big|\, \mathbf{X}^\ominus_n(t_n)=D}&=\P{\sum_{u\in \T_n^d\setminus D}\!\!\!\!\ind{X_{v^\oplus,u}^{\oplus,(D)}\leq t_n}<n-n^{1-\xi\delta_n}}\\ 
	&=\P{\sum_{u\in \T_n^d\setminus D}\!\!\!\!\ind{X_{v^\oplus,u}^{\oplus,(D)}>t_n}> n^{1-\xi\delta_n}-|D|}.
	\ea\ee 
	Here, we note that $n^{1-\xi\delta_n}-|D|$ tends to infinity with $n$, since $\delta_n=o(1)$ and $|D|\leq \log n$. By Markov's inequality, we then obtain the upper bound 
	\be 
	\P{\sum_{u\in \T_n^d\setminus D}\!\!\!\!\ind{X_{v^\oplus,u}^{\oplus,(D)}>t_n}> n^{1-\xi\delta_n}-|D|}\leq \frac{n-|D|}{n^{1-\xi\delta_n}-|D|}\P{X_{v^\oplus,U}^{\oplus,(D)}\geq t_n}, 
	\ee 
	where $U$ is a vertex selected uniformly at random from $\T_n^d\setminus D$. Using Proposition~\ref{prop:uniftime}, the probability can be bounded from above, uniformly in the set $D$ and the vertex $v^\oplus$, by $\e^{-K\delta_n \log n}$ for some $K>0$. The fraction $\frac{n-|D|}{n^{1-\xi\delta_n}-|D|}$ equals $(1+o(1))n^{\xi \delta_n}$, uniformly in the choice of $D$, so that choosing $\xi<K$ yields that 
    \be \label{eq:E2prob2}
    \mathbb P\Big(\big\{|\cN_n^\oplus(t_n)|<n-n^{1-\xi\delta_n}\big\} \cap\wt\cE_{1,n}\Big)=o(1).
    \ee 
    We also note that $\wt\cE_{1,n}\supseteq \cE_{1,n}$, so that $\mathbb P(\wt \cE_{1,n}^c)=o(1)$.  Combining this with~\eqref{eq:E2prob1} and~\eqref{eq:E2prob2}, we obtain that $\P{\cE_{2,n}^c}=o(1)$, as desired. 
	
	Finally, we turn to $\cE_{3,n}$. Recall the discrete-time description of the LRC process introduced prior to the proof of Lemma~\ref{lemma:phase2}. As discussed after~\eqref{eq:events}, the event $\cE_{3,n}$ implies that all vertices that remain uninfected at time $t_n$ (if any) are all reached by the $\oplus$ infection. We prove that $\P{\cE_{3,n}}$ tends to one using the discrete version of the process. We define
	\be 
	\cC_n:=\{|\cN_n^\ominus(t_n)|\leq \log n,\  |\cN_n^\oplus(t_n)|\geq n-n^{1-\xi\delta_n} \},
	\ee 
	and note that $ \cE_{1,n}\cap \cE_{2,n}\subseteq \cC_n$ (under the coupling of the LRC process to the branching random walks), as follows from~\eqref{eq:events}. We then bound
	\be\ba \label{eq:allremainprob}
	\P{\cE_{3,n}}&\geq \P{\cE_{3,n}\cap \cC_n}\\
	&=\mathbb P(\{\text{$\oplus$ infects all remaining }n-|\cN_n^\ominus(t_n)|-|\cN_n^\oplus(t_n)|\text{ vertices}\}\cap \cC_n).
	\ea\ee 
	We write $A_0:=\cN_n^\ominus(t_n)$ and let $B_i$ be the set of vertices infected by $\oplus$ when there are exactly $|\cN_n^\ominus(t_n)|+|\cN_n^\oplus(t_n)|+i$ many vertices infected, for $i\in\{0,\ldots, n-|\cN_n^\ominus(t_n)|-|\cN_n^\oplus(t_n)|\}$. In particular, $B_0=\cN_n^\oplus(t_n)$. We also set $\ell_n:=n-|\cN_n^\ominus(t_n)|-|\cN_n^\oplus(t_n)|=n-|A_0|-|B_0|$ and define the events $\cS_i:=\{|B_i|=|\cN_n^\oplus(t_n)|+i\}$ for $i\in[\ell_n]$. We can then write 
	\be\ba\label{eq:polyaproball}
	\mathbb P(\{\text{$\oplus$ infects all remaining }\ell_n\text{ vertices}\}\cap \cC_n)&=\P{\cC_n\cap\Big(\bigcap_{i=1}^{\ell_n}\cS_i\Big)}\\ 
	&=\E{\indicwo{\cC_n}\P{\bigcap_{i=1}^{\ell_n}\cS_i\,\Bigg|\, A_0,B_0}}.
	\ea\ee 
     We then bound 
     \be 
    \P{\bigcap_{i=1}^{\ell_n}\cS_i\,\Bigg|\, A_0,B_0}=\prod_{i=1}^{\ell_n}\P{\cS_i\,\big| A_0,B_0,(\cS_j)_{j<i}}\geq 1-\sum_{i=1}^{\ell_n}\P{\cS_i^c\,\big| A_0,B_0,(\cS_j)_{j<i}}.
     \ee 
    Each conditional probability in the sum can be written as 
    \be 
    \P{\cS_i^c\,\big| A_0,B_0,(\cS_j)_{j<i}}=\E{\frac{\sum_{v\in A_0}\sum_{u\in (A_0\cup B_{i-1})^c}\|u-v\|^{-\alpha_\ominus}}{\sum_{v\in B_{i-1}}\sum_{u\in (A_0\cup B_{i-1})^c}\lambda \|u-v\|^{-\alpha_\oplus}}\,\Bigg|\, A_0,B_0,(\cS_j)_{j<i}}.
    \ee 
    By conditioning on the events $(\cS_j)_{j<i}$, we have that $|B_{i-1}^c|=n-(|B_0|+(i-1))=\ell_n+|A_0|-(i-1)$ and $|(A_0\cup B_{i-1})^c|=n-|A_0|-|B_{i-1}|=\ell_n-(i-1)$. Using Lemma~\ref{lemma:rateest}, we can then bound, 
	\begin{align}
		\sum_{v\in A_0}\sum_{u\in (A_0\cup B_{i-1})^c}\|u-v\|^{-\alpha_\ominus}&\leq |A_0|R^\ominus_{|(A_0\cup B_{i-1})^c|} =|A_0|R^\ominus_{\ell_n-(i-1)},
		\shortintertext{and}
		\sum_{v\in B_{i-1}}\sum_{u\in (A_0\cup B_{i-1})^c}\lambda \|u-v\|^{-\alpha_\oplus}&\geq |(A_0\cup B_{i-1})^c|(R_n^\oplus-R^\oplus_{|B_{i-1}^c|})\\ 
		&=(\ell_n-(i-1))(R_n^\oplus-R^\oplus_{\ell_n+|A_0|-(i-1)}).
	\end{align}
	This yields
	\be 
	\mathbb P\Bigg(\bigcap_{i=1}^{\ell_n}\cS_i\,\Bigg|\, A_0,B_0\Bigg)
	\geq 1-\sum_{i=1}^{\ell_n}\frac{|A_0|R^\ominus_{\ell_n-(i-1)}}{(\ell_n-(i-1))(R_n^\oplus-R^\oplus_{\ell_n+|A_0|-(i-1)})}.
	\ee 
    On the event $\cC_n$, we can bound $|A_0|\leq \log n$ and $\ell_n=n-|A_0|-|B_0|\leq n-|B_0|\leq n^{1-\xi\delta_n}$. As $R_n^\oplus$ is increasing in $n$, we obtain by setting $j=\ell_n-(i-1)$ the lower bound 
    \be 
    \E{\indicwo{\cC_n}\P{\bigcap_{i=1}^{\ell_n}\cS_i\,\Bigg|\, A_0,B_0}}\geq \P{\cC_n}\bigg(1-\sum_{j=1}^{n^{1-\xi\delta_n}}\frac{\log(n)R^\ominus_j}{j(R^\oplus_n-R^\oplus_{j+\log n})}\bigg).
    \ee 
    Since $\delta_n$ is such that $\delta_n\log n$ diverges, it follows that $n^{1-\xi\delta_n}+\log n=o(n)$, for any $\xi>0$. Hence, $R^\oplus_n-R^\oplus_{i+\log n}=R_n^\oplus(1-o(1))=Z_nR_n^\ominus(1-o(1))$. Moreover, we can bound $\underline c j^{1-\alpha_\ominus/d}\leq R_j^\ominus\leq \overline c j^{1-\alpha_\ominus/d}$ for some constants $\underline c,\overline c>0$, which are independent of $j$, using Lemma~\ref{lemma:ratesupinf}. As a result, we have the lower bound
	\be 
	\P{\cC_n}\bigg(1-\frac{(\overline c+o(1))\log n}{Z_n \underline c n^{1-\alpha_\ominus/d}}\sum_{i=1}^{n^{1-\xi\delta_n}}i^{-\alpha_\ominus/d}\bigg)\geq \P{\cC_n}\Big(1-C\frac{\log n}{Z_n}n^{-(1-\alpha_\ominus/d)\xi\delta_n}\Big), 
	\ee 
    where we bound the sum from above by an integral and  where $C>0$ is some sufficiently large constant. Since $\sup_{n\in\N}\alpha_\ominus(n)<d$, $\delta_n\log n$ diverges, and $Z_n=(c+o(1))\log n$, we find that this lower bound equals $\P{\cC_n}(1-o(1))$. Since $\cE_{1,n}\cap \cE_{2,n}\subseteq \cC_n$ and we already proved that $\P{\cE_{1,n}\cap \cE_{2,n}}=1-o(1)$, we thus finally use~\eqref{eq:allremainprob} to derive that $\P{\cE_{3,n}}=1-o(1)$, which concludes the proof.
\end{proof}

\subsection{Lower bound for $N_n^\ominus$}\label{sec:lb}

In this section we prove a lower bound for the size of the $\ominus$ infection in Theorem~\ref{thrm:nocoex}~\ref{case:liminf} and~\ref{case:div}. This is summarised in the following proposition.

\begin{proposition}\label{prop:lbinfection}
	Fix $d\in\N$, let $\alpha_\ominus,\alpha_\oplus\in[0,d)$ such that $\sup_{n\in\N}\alpha_\square(n)<d$ for $\square\in\{\ominus, \oplus\}$, and recall $Z_n$ from~\eqref{eq:Zn}. Let $N_n^\ominus$ denote the final size of the $\ominus$ infection of an LRC process, started from two distinct initial vertices $v^\ominus,v^\oplus\in\T_n^d$. Fix $\xi\in(0,1)$ and define $\eps_n:=\xi(\frac{c_n}{\log n}\wedge 1)$. When $c_n=o((\log n)^2)$ and $c_n$ diverges, the event $\{N_n^\ominus\geq n^{(1-(1+\xi)\eps_n)/Z_n}\}$ holds with high probability.
\end{proposition}

\begin{proof}
	We couple the LRC process with two branching random walks, as in Section~\ref{sec:coupling}. We define $t_n:=(1-\eps_n)\log( n)/Z_n$, and 
	\be 
	\cA_n^\ominus:=\{|\cO_n^\ominus(t_n)|\geq n^{(1-(1+\xi)\eps_n)/Z_n}\}.
	\ee 
	$\cA_n^\ominus$ implies the result, since the size of the $\ominus$ infection is monotone increasing in time, so that $N_n^\ominus\geq |\cO_n^\ominus(t_n)|\geq n^{(1-(1+\xi)\eps_n)/Z_n}$. We thus show that $\P{\cA_n^\ominus}\to1$.
    
	Since $\limsup_{n\to\infty}\eps_n<1$ and $c_n$ diverges, we can apply Proposition~\ref{prop:gen_size_artificial_CTBP_tail_bound} with $\delta_n=\eps_n$. We use~\eqref{eq:Csublin} from Proposition~\ref{prop:gen_size_artificial_CTBP_tail_bound}, which states that $\cD_n^\ominus(t_n)\e^{-t_n}\toinp 0$. Since $Z_n=o(\log n)$ (that is, $c_n=o((\log n)^2$), it follows that  $t_n$ diverges with $n$, so that 
	\be 
	\e^{-t_n}|\cX_n^\ominus(t_n)|=\e^{-t_n}|\BP^\ominus(t_n)|\toas E_1, 
	\ee 
    by Lemma~\ref{lem:size_CTBP}, where $E_1$ is a rate-one exponential random variable. Combined, we obtain
	\be 
	\e^{-t_n}|\cO_n^\ominus(t_n)|= \e^{-t_n}\big(|\cX_n^\ominus(t_n)|-\cD_n^\ominus(t_n)\big)\toinp E_1.
	\ee 
	As a result, since $c_n$ diverges but $c_n=o((\log n)^2)$, and thus
	\be 
	n^{(1-(1+\xi)\eps_n)/Z_n}=\e^{t_n}n^{-\xi\eps_n/Z_n}=\exp\Big(t_n -\xi^2\frac{(c_n\wedge \log n)\log n}{\log n+c_n}\Big)=o\big(\e^{t_n}\big), 
	\ee 
	it follows that the event $\cA_n^\ominus$ holds with high probability, which concludes the proof. 
\end{proof}

\noindent We conclude this section by proving Theorem~\ref{thrm:nocoex}.

\begin{proof}[Proof of Theorem~\ref{thrm:nocoex}]   The proofs of Cases~\ref{case:1} and~\ref{case:fin} follow from Sections~\ref{sec:winnerone} and~\ref{sec:winnerfinite}, respectively. To prove Cases~\ref{case:liminf} and~\ref{case:div}, fix $\eps>0$ and let $A=A(\eps)$ be a sufficiently large constant, such that $\P{N_n^\ominus\geq An^{1/Z_n}}<\eps$, which is possible by Proposition~\ref{prop:infectsizes}. Then, in Case~\ref{case:liminf}, with probability at least $1-\eps$, 
	\be \label{eq:logNub}
	\frac{\log(N_n^\ominus)Z_n}{\log n}\leq 1+\log(A)\frac{Z_n}{\log n}. 
	\ee 
	Since $c_n=o((\log n)^2)$ in Case~\ref{case:liminf}, so that $Z_n=o(\log n)$, it follows that the right-hand side is smaller than $1+\xi$ for $n$ sufficiently large with probability at least $1-\eps$, for any $\xi>0$. Since we can take $\eps$ arbitrarily small, independently of $\xi$, it follows that $\log(N_n^\ominus)Z_n/\log n$ is at most $1+\xi$ with high probability for any $\xi>0$. To prove a matching lower bound, we use Proposition~\ref{prop:lbinfection} to conclude that $N_n^\ominus\geq n^{(1-(1+\zeta)\eps_n)/Z_n}$, where $\eps_n:=\zeta(\frac{c_n}{\log n}\wedge 1)$, holds with high probability for any $\zeta>0$. By taking the logarithm on both sides and rearranging terms,
	\be \label{eq:Nlb}
	\log(N_n^\ominus)\frac{Z_n}{\log n} \geq 1-(1+\zeta)\eps_n. 
	\ee 
	As $\eps_n\leq \zeta$, the right-hand side is at least $1-(1+\zeta)\zeta=1-\xi$ by letting $\xi\coloneqq\zeta (1+\zeta)$. Since $\zeta$ is arbitrary, we obtain that $\log(N_n^\ominus)Z_n/\log n$ is at least $1-\xi$ with high probability  for any $\xi>0$. Combined with the upper bound, this concludes the proof of Case~\ref{case:liminf}. 
	
	In Case~\ref{case:div} we can be more precise in~\eqref{eq:logNub} by observing that $Z_n=1+c_n/\log n$, so that 
	\be
	c_n^{-1}(\log n-\log N_n^\ominus) \geq \frac{\log n}{\log n+c_n}-\frac{\log A}{c_n}. 
	\ee 
	Since in Case~\ref{case:div} $c_n$ diverges such that $c_n=o(\log n)$, it follows that the right-hand side is at least $1-\xi$ for $n$ sufficiently large. As the inequality holds with probability at least $1-\eps$, where we can make $\eps$ arbitrarily small, we obtain that $c_n^{-1}(\log n-\log N_n^\ominus)\geq 1-\xi$ with high probability for any $\xi>0$. We then rewrite~\eqref{eq:Nlb} by using that $Z_n=1+c_n/\log n$ and that $\eps_n=\zeta c_n/\log n$ for all large $n$, since $c_n=o(\log n)$. As a result, we  obtain, with high probability, 
	\be 
	c_n^{-1}(\log n-\log N_n^\ominus) \leq \frac{\log n}{\log n+c_n}(1+\zeta(1+\zeta))\leq 1+\xi,  
	\ee 
	when we again set $\xi:=\zeta(1+\zeta)$. Since $\zeta$ is arbitrary, we obtain a matching upper bound, which proves Case~\ref{case:div} and concludes the proof.
\end{proof}

\subsection{Coexistence vs.\ absence of coexistence: Proof of Theorem~\ref{thrm:coex}}\label{sec:coex}

\begin{proof}[Proof of Theorem~\ref{thrm:coex}]
	We first prove that coexistence occurs when $\sup_{n\in\N}|c_n|<\infty$. We assume without loss of generality that $Z_n\geq 1$, which is equivalent to $c_n\geq 0$.  Our  aim is to check \eqref{eq:def_coex} for $\square\in\{\ominus, \oplus\}$. We use the coupling introduced in Section~\ref{sec:coupling}, and hence analyse the minimum of $|\cO_n^\ominus(t)|$ and $|\cO_n^\oplus(t)|$. The desired result is then implied by the following: There exists a function $f:(0,1)\to\R_+$ such that $\lim_{x\downarrow 0}f(x)=0$, such that for any $\varepsilon>0$ sufficiently small,
	\be\label{eq:RTP_coexistence}
	\liminf_{n \to \infty}\Prob{\min\Big\{\frac{|\cO_n^\ominus(T_n)|}{n},\frac{|\cO_n^\oplus(T_n)|}{n}\Big\}\geq f(\eps)}\geq 1-\eps,
	\ee 
	and some suitable time $T_n$. Here, we note that the $T_n$ need not be $T_{\mathrm{cov}}$ (the time when all vertices in $\T_n^d$ contain an original particle), since particles stay where they are and do not die and thus $|\cO_n^\square(t)|$ is non-decreasing in $t$. Let us fix $\varepsilon>0$ and take $m>0$. Consider the events
	\begin{align*}
		\mathcal{G}_1(m)&:=\{\inf_{t> 0}|\cX_n^\ominus(t)|\e^{-t}>m^{-1}\}\cap\{\inf_{t>0}|\cX_n^\oplus(t)|\e^{-Z_n t}>m^{-1}\}.
	\end{align*}
	By Lemma~\ref{lem:size_CTBP} and since $|\cX_n^\square(t)|=|\BP^\square(t)|$, we can choose an $M_1=M_1(\varepsilon)$ large enough such that $\Prob{\mathcal{G}_1(m)}>1-\varepsilon/2$ for all $m\geq M_1$. Further, let us fix $\delta \in (0,\inf_{n\in\N}\min\{1-\alpha_\ominus(n)/d,1-\alpha_\oplus(n)/d\})$. By using Proposition~\ref{prop:gen_size_artificial_CTBP_tail_bound}, we can choose $M_2=M_2(\varepsilon,\delta,C)$ large enough so that for all $m\geq M_2$, 
	\be 
    \cG_2(m):=\bigcap_{\square\in\{\ominus, \oplus\}}\Big\{\cD_n^\square\Big( \frac{\log n-m}{Z_n}\Big)\leq n\e^{-m(1+\delta)}\Big\} 
	\ee 
	holds with probability at least $1-\eps/2$, and that $\e^{-(m+C)}m^{-1}>\e^{-m(1+\delta)}$  holds for all $m\geq M_2$, where $C>0$ is a constant to be determined. Consider $M=M(\varepsilon)=\max\{M_1,M_2\}$ and define 
    \be 
    T_n:=\frac{\log{n}-M}{Z_n},
    \ee 
    We note that, since $Z_n=1+c_n/\log n$ and $c_n$ is bounded, that there exists $C>0$ such that 
    \be \label{eq:Tndif}
    \sup_{n\in\N}|T_n-(\log n-M)|<C.
    \ee
    Finally, set $f(\eps):=\e^{-(M+C)}M^{-1}-\e^{-M(1+\delta)}$ and let $\mathcal{G}:=\mathcal{G}_1(M)\cap \mathcal{G}_2(M)$. We show that the event $\cG$ implies the desired bound $|\cO_n^\square(T_n)|/n\geq f(\eps)$ for both $\square=\ominus$ and $\square=\oplus$. Let us start with $\square=\ominus$. First, by the event $\cG_1(M)$, we have $|\BP^\ominus(T_n)|>\e^{T_n}M^{-1}>n\e^{-(M+C)}M^{-1}$, where the final step uses~\eqref{eq:Tndif}. Second, by the event $\cG_2(M)$, we have 
	\be 
	|\cD^\ominus(T_n)| \leq n \e^{-M(1+\delta)}. 
	\ee 
	Combined, we obtain  
	\be \label{eq:NTNlb}
	|\cO_n^\ominus(T_n)|=|\cX_n^\ominus(T_n)|-|\cD^\ominus(T_n)|\geq n\big(\e^{-(M+C)}M^{-1}-\e^{-M(1+\delta)}\big)=nf(\eps). 
	\ee 
    For $\square=\oplus$, we similarly have that
	\be 
	|\cX_n^\oplus(T_n)|\geq M^{-1}\e^{Z_nT_n}=nM^{-1}\e^{-M},
	\ee 
    which yields the same conclusion in~\eqref{eq:NTNlb}, but for the $\oplus$ infection, so that also $ |\cO_n^\oplus(T_n)|\geq nf(\eps)$ is implied by $\cG$. Hence, we arrive at 
	\be 
	\liminf_{n\to\infty}\P{\min\Big\{\frac{|\cO_n^\ominus(T_n)|}{n},\frac{|\cO_n^\oplus(T_n)|}{n}\Big\}\geq f(\eps)}\geq \liminf_{n\to\infty}\P{\cG}\geq 1-\eps,
	\ee 
	as desired. 

    We then prove that absence of coexistence, as in~\eqref{eq:def_abs_coex}, occurs when $\lim_{n\to\infty}|c_n|=\infty$. As we assume $c_n\geq 0$, or equivalently $Z_n\geq 1$, we are required to prove that $N_n^\ominus/n \toinp 0$. This directly follows from the fact that $(N_n^\ominus n^{-1/Z_n})_{n\in\N}$ is a tight sequence by Proposition~\ref{prop:infectsizes}, and 
	\be 
	n^{1/Z_n}=n^{1-(Z_n-1)/Z_n}=n\exp\Big(-\frac{c_n\log n}{\log n+c_n}\Big)=o(n), 
	\ee 
	where the final step holds since $c_n$ diverges. Hence, we have $N_n^\ominus/n\toinp 0$, and we conclude the proof.
\end{proof}

    \paragraph{Acknowledgements.} 

BL has received funding from the European Union’s Horizon 2022 research and innovation programme under the Marie Sk\l{}odowska-Curie grant agreement no.~$101108569$, ``DynaNet", and has been supported by the grant GrHyDy ANR-20-CE40-0002.  

Initial ideas for this project were conceived during the MATRIX-MFO Tandem Workshop `Stochastic Reinforcement Processes and Graphs', and the authors thank the organizers Markus Heydenreich and C\'ecile Mailler for the invitation, as well as MFO Oberwolfach for the hospitality. BL and NM thank Mia Deijfen, Markus Heydenreich, and Remco van der Hofstad for interesting discussions and for useful suggestions, and Remco van der Hofstad and Partha Dey for a careful reading of a first draft.

This paper forms part of NM's PhD thesis.

We also thank two anonymous referees for their suggestions that helped to improve the presentation of the paper.

\bibliographystyle{abbrvnat}
\bibliography{ref}

\providecommand{\vander}{van der }
\begin{thebibliography}{48}
\providecommand{\natexlab}[1]{#1}
\providecommand{\url}[1]{\texttt{#1}}
\expandafter\ifx\csname urlstyle\endcsname\relax
  \providecommand{\doi}[1]{doi: #1}\else
  \providecommand{\doi}{doi: \begingroup \urlstyle{rm}\Url}\fi

\bibitem[Addario-Berry et~al.(2010)Addario-Berry, Broutin, and
  Lugosi]{AddBrouLug10}
L.~Addario-Berry, N.~Broutin, and G.~Lugosi.
\newblock The longest minimum-weight path in a complete graph.
\newblock \emph{Combinatorics, Probability and Computing}, 19\penalty0
  (1):\penalty0 1–19, 2010.
\newblock \doi{10.1017/S0963548309990204}.

\bibitem[Adriaans and Komj{\'a}thy(2018)]{AdrKom18}
E.~Adriaans and J.~Komj{\'a}thy.
\newblock Weighted distances in scale-free configuration models.
\newblock \emph{Journal of Statistical Physics}, 173:\penalty0 1082--1109,
  2018.

\bibitem[Ahlberg et~al.(2019)Ahlberg, Deijfen, and
  Janson]{Svante_CM_fin_var_19}
D.~Ahlberg, M.~Deijfen, and S.~Janson.
\newblock Competing first passage percolation on random graphs with finite
  variance degrees.
\newblock \emph{Random Structures \& Algorithms}, 55\penalty0 (3):\penalty0
  545--559, 2019.
\newblock \doi{https://doi.org/10.1002/rsa.20846}.
\newblock URL \url{https://onlinelibrary.wiley.com/doi/abs/10.1002/rsa.20846}.

\bibitem[Ahlberg et~al.(2020)Ahlberg, Deijfen, and Hoffman]{ahlberg2020two}
D.~Ahlberg, M.~Deijfen, and C.~Hoffman.
\newblock {The two-type {R}ichardson model in the half-plane}.
\newblock \emph{The Annals of Applied Probability}, 30\penalty0 (5):\penalty0
  2261 -- 2273, 2020.
\newblock \doi{10.1214/19-AAP1557}.
\newblock URL \url{https://doi.org/10.1214/19-AAP1557}.

\bibitem[Antunović et~al.(2017)Antunović, Dekel, Mossel, and
  Peres]{Per_RRG_17}
T.~Antunović, Y.~Dekel, E.~Mossel, and Y.~Peres.
\newblock Competing first passage percolation on random regular graphs.
\newblock \emph{Random Structures \& Algorithms}, 50\penalty0 (4):\penalty0
  534--583, 2017.
\newblock \doi{https://doi.org/10.1002/rsa.20699}.
\newblock URL \url{https://onlinelibrary.wiley.com/doi/abs/10.1002/rsa.20699}.

\bibitem[Auffinger et~al.(2017)Auffinger, Damron, and Hanson]{auffinger201750}
A.~Auffinger, M.~Damron, and J.~Hanson.
\newblock \emph{50 years of first-passage percolation}, volume~68.
\newblock American Mathematical Soc., 2017.

\bibitem[Banerjee and Bhamidi(2022)]{BanBha22}
S.~Banerjee and S.~Bhamidi.
\newblock Root finding algorithms and persistence of jordan centrality in
  growing random trees.
\newblock \emph{The Annals of Applied Probability}, 32\penalty0 (3):\penalty0
  2180--2210, 2022.

\bibitem[Baroni et~al.(2015)Baroni, {\vander H}ofstad, and
  Komj{\'a}thy]{BarHofKom15}
E.~Baroni, R.~{\vander H}ofstad, and J.~Komj{\'a}thy.
\newblock {Fixed speed competition on the configuration model with infinite
  variance degrees: unequal speeds}.
\newblock \emph{Electronic Journal of Probability}, 20\penalty0
  (none):\penalty0 1 -- 48, 2015.
\newblock \doi{10.1214/EJP.v20-3749}.
\newblock URL \url{https://doi.org/10.1214/EJP.v20-3749}.

\bibitem[Baroni et~al.(2017)Baroni, {\vander H}ofstad, and
  Komj{\'a}thy]{BarHofKom17}
E.~Baroni, R.~{\vander H}ofstad, and J.~Komj{\'a}thy.
\newblock Nonuniversality of weighted random graphs with infinite variance
  degree.
\newblock \emph{Journal of Applied Probability}, 54\penalty0 (1):\penalty0
  146--164, 2017.

\bibitem[Baroni et~al.(2019)Baroni, {\vander H}ofstad, and
  Komj{\'a}thy]{BarHofKom19}
E.~Baroni, R.~{\vander H}ofstad, and J.~Komj{\'a}thy.
\newblock Tight fluctuations of weight-distances in random graphs with
  infinite-variance degrees.
\newblock \emph{Journal of Statistical Physics}, 174:\penalty0 906--934, 2019.

\bibitem[Bhamidi and {\vander H}ofstad(2012)]{BhaHof12}
S.~Bhamidi and R.~{\vander H}ofstad.
\newblock Weak disorder asymptotics in the stochastic mean-field model of
  distance.
\newblock \emph{The Annals of Applied Probability}, 22\penalty0 (1):\penalty0
  29--69, 2012.

\bibitem[Bhamidi and {\vander H}ofstad(2017)]{BhaHof17}
S.~Bhamidi and R.~{\vander H}ofstad.
\newblock Diameter of the stochastic mean-field model of distance.
\newblock \emph{Combinatorics, Probability and Computing}, 26\penalty0
  (6):\penalty0 797--825, 2017.

\bibitem[Bhamidi et~al.(2010)Bhamidi, {\vander H}ofstad, and
  Hooghiemstra]{BhaHofHoog10.2}
S.~Bhamidi, R.~{\vander H}ofstad, and G.~Hooghiemstra.
\newblock {First passage percolation on random graphs with finite mean
  degrees}.
\newblock \emph{The Annals of Applied Probability}, 20\penalty0 (5):\penalty0
  1907 -- 1965, 2010.
\newblock \doi{10.1214/09-AAP666}.
\newblock URL \url{https://doi.org/10.1214/09-AAP666}.

\bibitem[Bhamidi et~al.(2011)Bhamidi, {\vander H}ofstad, and
  Hooghiemstra]{BhaHofHoog11}
S.~Bhamidi, R.~{\vander H}ofstad, and G.~Hooghiemstra.
\newblock First passage percolation on the {E}rd{\H{o}}s--{R}{\'e}nyi random
  graph.
\newblock \emph{Combinatorics, Probability and Computing}, 20\penalty0
  (5):\penalty0 683--707, 2011.

\bibitem[Bhamidi et~al.(2013)Bhamidi, {\vander H}ofstad, and
  Hooghiemstra]{BhaHofHoog13}
S.~Bhamidi, R.~{\vander H}ofstad, and G.~Hooghiemstra.
\newblock Weak disorder in the stochastic mean-field model of distance {II}.
\newblock \emph{Bernoulli}, 19\penalty0 (2):\penalty0 363--386, 2013.

\bibitem[Bhamidi et~al.(2017)Bhamidi, {\vander H}ofstad, and
  Hooghiemstra]{BhaHofHoog17}
S.~Bhamidi, R.~{\vander H}ofstad, and G.~Hooghiemstra.
\newblock {Universality for first passage percolation on sparse random graphs}.
\newblock \emph{The Annals of Probability}, 45\penalty0 (4):\penalty0 2568 --
  2630, 2017.
\newblock \doi{10.1214/16-AOP1120}.
\newblock URL \url{https://doi.org/10.1214/16-AOP1120}.

\bibitem[Chatterjee and S.~Dey(2016{\natexlab{a}})]{ChaDey16}
S.~Chatterjee and P.~S.~Dey.
\newblock Multiple phase transitions in long-range first-passage percolation on
  square lattices.
\newblock \emph{Communications on Pure and Applied Mathematics}, 69\penalty0
  (2):\penalty0 203--256, 2016{\natexlab{a}}.

\bibitem[Chatterjee and S.~Dey(2016{\natexlab{b}})]{SC_PD_LRFPP}
S.~Chatterjee and P.~S.~Dey.
\newblock Multiple phase transitions in long-range first-passage percolation on
  square lattices.
\newblock \emph{Communications on Pure and Applied Mathematics}, 69\penalty0
  (2):\penalty0 203--256, 2016{\natexlab{b}}.
\newblock \doi{https://doi.org/10.1002/cpa.21571}.
\newblock URL \url{https://onlinelibrary.wiley.com/doi/abs/10.1002/cpa.21571}.

\bibitem[Coletti et~al.(2023)Coletti, De~Lima, Hinsen, Jahnel, and
  Valesin]{ColLimHinJahVal23}
C.~F. Coletti, L.~R. De~Lima, A.~Hinsen, B.~Jahnel, and D.~Valesin.
\newblock Limiting shape for first-passage percolation models on random
  geometric graphs.
\newblock \emph{Journal of Applied Probability}, 60\penalty0 (4):\penalty0
  1367--1385, 2023.

\bibitem[Deijfen and H{\"a}ggstr{\"o}m(2008)]{mia_survey}
M.~Deijfen and O.~H{\"a}ggstr{\"o}m.
\newblock The pleasures and pains of studying the two-type {R}ichardson model.
\newblock \emph{Analysis and Stochastics of Growth Processes and interface
  models}, 7:\penalty0 10--1093, 2008.

\bibitem[Deijfen and van~der Hofstad(2016)]{Remco_Mia_WTA}
M.~Deijfen and R.~van~der Hofstad.
\newblock The winner takes it all.
\newblock \emph{The Annals of Applied Probability}, 26\penalty0 (4):\penalty0
  2419--2453, 2016.
\newblock ISSN 10505164.
\newblock URL \url{http://www.jstor.org/stable/24810061}.

\bibitem[Deijfen et~al.(2023)Deijfen, {\vander H}ofstad, and
  Sfragara]{deijfen_vanderhofstad_sfragara_2023}
M.~Deijfen, R.~{\vander H}ofstad, and M.~Sfragara.
\newblock The winner takes it all but one.
\newblock \emph{Journal of Applied Probability}, page 1–16, 2023.
\newblock \doi{10.1017/jpr.2023.23}.

\bibitem[Eckhoff et~al.(2013)Eckhoff, Goodman, {\vander H}ofstad, and
  Nardi]{EckGoodHofNar13}
M.~Eckhoff, J.~Goodman, R.~{\vander H}ofstad, and F.~R. Nardi.
\newblock Short paths for first passage percolation on the complete graph.
\newblock \emph{Journal of Statistical Physics}, 151\penalty0 (6):\penalty0
  1056--1088, 2013.

\bibitem[Eckhoff et~al.(2020{\natexlab{a}})Eckhoff, Goodman, {\vander H}ofstad,
  and Nardi]{EckGoodHofNar15.1}
M.~Eckhoff, J.~Goodman, R.~{\vander H}ofstad, and F.~R. Nardi.
\newblock {Long paths in first passage percolation on the complete graph {I}.
  Local {PWIT} dynamics}.
\newblock \emph{Electronic Journal of Probability}, 25:\penalty0 1 -- 45,
  2020{\natexlab{a}}.
\newblock \doi{10.1214/20-EJP484}.
\newblock URL \url{https://doi.org/10.1214/20-EJP484}.

\bibitem[Eckhoff et~al.(2020{\natexlab{b}})Eckhoff, Goodman, {\vander H}ofstad,
  and Nardi]{EckGoodHofNar15.2}
M.~Eckhoff, J.~Goodman, R.~{\vander H}ofstad, and F.~R. Nardi.
\newblock Long paths in first passage percolation on the complete graph {II}.
  {G}lobal branching dynamics.
\newblock \emph{Journal of Statistical Physics}, 181\penalty0 (2):\penalty0
  364--447, 2020{\natexlab{b}}.

\bibitem[Hadj~Hassine(2022)]{Hadj22}
I.~Hadj~Hassine.
\newblock Covid-19 vaccines and variants of concern: A review.
\newblock \emph{Reviews in Medical Virology}, 32\penalty0 (4):\penalty0 e2313,
  2022.
\newblock \doi{https://doi.org/10.1002/rmv.2313}.
\newblock URL \url{https://onlinelibrary.wiley.com/doi/abs/10.1002/rmv.2313}.

\bibitem[Hammersley and Welsh(1965)]{Hammersley1965}
J.~M. Hammersley and D.~J.~A. Welsh.
\newblock \emph{First-Passage Percolation, Subadditive Processes, Stochastic
  Networks, and Generalized Renewal Theory}, pages 61--110.
\newblock Springer Berlin Heidelberg, Berlin, Heidelberg, 1965.
\newblock ISBN 978-3-642-99884-3.
\newblock \doi{10.1007/978-3-642-99884-3_7}.
\newblock URL \url{https://doi.org/10.1007/978-3-642-99884-3_7}.

\bibitem[{\vander H}ofstad and Komj{\'a}thy(2015)]{BarHofKom15.2}
R.~{\vander H}ofstad and J.~Komj{\'a}thy.
\newblock Fixed speed competition on the configuration model with infinite
  variance degrees: equal speeds.
\newblock \emph{arXiv preprint arXiv:1503.09046}, 2015.

\bibitem[{\vander H}ofstad and Komj{\'a}thy(2017)]{HofKom17}
R.~{\vander H}ofstad and J.~Komj{\'a}thy.
\newblock Explosion and distances in scale-free percolation.
\newblock \emph{arXiv preprint arXiv:1706.02597}, 2017.

\bibitem[{\vander H}ofstad et~al.(2001){\vander H}ofstad, Hooghiemstra, and
  Van~Mieghem]{HofHoogMie01}
R.~{\vander H}ofstad, G.~Hooghiemstra, and P.~Van~Mieghem.
\newblock First-passage percolation on the random graph.
\newblock \emph{Probability in the Engineering and Informational Sciences},
  15\penalty0 (2):\penalty0 225–237, 2001.
\newblock \doi{10.1017/S026996480115206X}.

\bibitem[{\vander H}ofstad et~al.(2002){\vander H}ofstad, Hooghiemstra, and
  Van~Mieghem]{HofHoogMieg02}
R.~{\vander H}ofstad, G.~Hooghiemstra, and P.~Van~Mieghem.
\newblock The flooding time in random graphs.
\newblock \emph{Extremes}, 5\penalty0 (2):\penalty0 111--129, 2002.

\bibitem[{\vander H}ofstad et~al.(2006){\vander H}ofstad, Hooghiemstra, and
  Van~Mieghem]{HofHoogMie06}
R.~{\vander H}ofstad, G.~Hooghiemstra, and P.~Van~Mieghem.
\newblock Size and weight of shortest path trees with exponential link weights.
\newblock \emph{Combinatorics, Probability and Computing}, 15\penalty0
  (6):\penalty0 903–926, 2006.
\newblock \doi{10.1017/S0963548306007802}.

\bibitem[Hooghiemstra and Van~Mieghem(2008)]{HoogMie08}
G.~Hooghiemstra and P.~Van~Mieghem.
\newblock The weight and hopcount of the shortest path in the complete graph
  with exponential weights.
\newblock \emph{Combinatorics, Probability and Computing}, 17\penalty0
  (4):\penalty0 537–548, 2008.
\newblock \doi{10.1017/S0963548308009176}.

\bibitem[Häggström and Pemantle(1998)]{Hag_Pem_98}
O.~Häggström and R.~Pemantle.
\newblock First passage percolation and a model for competing spatial growth.
\newblock \emph{Journal of Applied Probability}, 35\penalty0 (3):\penalty0
  683--692, 1998.
\newblock ISSN 00219002.
\newblock URL \url{http://www.jstor.org/stable/3215643}.

\bibitem[Häggström and Pemantle(2000)]{Hagg_Pem_00}
O.~Häggström and R.~Pemantle.
\newblock Absence of mutual unbounded growth for almost all parameter values in
  the two-type {R}ichardson model.
\newblock \emph{Stochastic Processes and their Applications}, 90\penalty0
  (2):\penalty0 207--222, 2000.
\newblock ISSN 0304-4149.
\newblock \doi{https://doi.org/10.1016/S0304-4149(00)00042-9}.
\newblock URL
  \url{https://www.sciencedirect.com/science/article/pii/S0304414900000429}.

\bibitem[Janson(1999)]{Jans99}
S.~Janson.
\newblock One, two and three times log n/n for paths in a complete graph with
  random weights.
\newblock \emph{Combinatorics, Probability and Computing}, 8\penalty0
  (4):\penalty0 347--361, 1999.

\bibitem[Janson(2006)]{Janson2006}
S.~Janson.
\newblock Limit theorems for triangular urn schemes.
\newblock \emph{Probability Theory and Related Fields}, 134\penalty0
  (3):\penalty0 417--452, Mar 2006.
\newblock ISSN 1432-2064.
\newblock \doi{10.1007/s00440-005-0442-7}.
\newblock URL \url{https://doi.org/10.1007/s00440-005-0442-7}.

\bibitem[Jorritsma and Komj{\'a}thy(2020)]{JorKom20}
J.~Jorritsma and J.~Komj{\'a}thy.
\newblock Weighted distances in scale-free preferential attachment models.
\newblock \emph{Random Structures \& Algorithms}, 57\penalty0 (3):\penalty0
  823--859, 2020.

\bibitem[Jorritsma and Komj{\'a}thy(2022)]{JorKom22}
J.~Jorritsma and J.~Komj{\'a}thy.
\newblock Distance evolutions in growing preferential attachment graphs.
\newblock \emph{The Annals of Applied Probability}, 32\penalty0 (6):\penalty0
  4356--4397, 2022.

\bibitem[Komj{\'a}thy and Lodewijks(2020)]{KomLod20}
J.~Komj{\'a}thy and B.~Lodewijks.
\newblock Explosion in weighted hyperbolic random graphs and geometric
  inhomogeneous random graphs.
\newblock \emph{Stochastic Processes and their Applications}, 130\penalty0
  (3):\penalty0 1309--1367, 2020.

\bibitem[Komj{\'a}thy et~al.(2021)Komj{\'a}thy, Lapinskas, and
  Lengler]{KomLapLen21}
J.~Komj{\'a}thy, J.~Lapinskas, and J.~Lengler.
\newblock Penalising transmission to hubs in scale-free spatial random graphs.
\newblock In \emph{Annales de l'Institut Henri Poincare (B) Probabilites et
  statistiques}, volume~57, pages 1968--2016. Institut Henri Poincar{\'e},
  2021.

\bibitem[Komj{\'a}thy et~al.(2024)Komj{\'a}thy, Lapinskas, Lengler, and
  Schaller]{KomLapLenSchal23_2}
J.~Komj{\'a}thy, J.~Lapinskas, J.~Lengler, and U.~Schaller.
\newblock Polynomial growth in degree-dependent first passage percolation on
  spatial random graphs.
\newblock \emph{Electronic Journal of Probability}, 29:\penalty0 1--48, 2024.

\bibitem[Komj{\'a}thy et~al.(2026)Komj{\'a}thy, Lapinskas, Lengler, and
  Schaller]{KomLapLenSchal23_1}
J.~Komj{\'a}thy, J.~Lapinskas, J.~Lengler, and U.~Schaller.
\newblock Four universal growth regimes in degree-dependent first passage
  percolation on spatial random graphs.
\newblock In \emph{Forum of Mathematics, Sigma}, volume~14, page e22. Cambridge
  University Press, 2026.

\bibitem[Last and Penrose(2017)]{last_penrose_2017}
G.~Last and M.~Penrose.
\newblock \emph{Lectures on the Poisson Process}.
\newblock Institute of Mathematical Statistics Textbooks. Cambridge University
  Press, 2017.
\newblock \doi{10.1017/9781316104477}.

\bibitem[Norris(1998)]{Nor98}
J.~R. Norris.
\newblock \emph{Markov chains}.
\newblock Number~2. Cambridge university press, 1998.

\bibitem[van~der Hofstad(2026)]{van2023universalwinner}
R.~van~der Hofstad.
\newblock Universal ‘winner-takes-it-all’ phenomenon in scale-free random
  graphs.
\newblock In \emph{Annales de l'Institut Henri Poincare (B) Probabilites et
  statistiques}, volume~62, pages 86--109. Institut Henri Poincar{\'e}, 2026.

\bibitem[van~der Hofstad and Lodewijks(2024)]{HofLod23}
R.~van~der Hofstad and B.~Lodewijks.
\newblock Long-range first-passage percolation on the torus.
\newblock \emph{Journal of Statistical Physics}, 191\penalty0 (9):\penalty0
  107, 2024.

\bibitem[Zhao et~al.(2022)Zhao, Huang, Zhang, Chen, Gao, and
  Jiao]{ZhaHuaZhaCheGaoJia22}
Y.~Zhao, J.~Huang, L.~Zhang, S.~Chen, J.~Gao, and H.~Jiao.
\newblock The global transmission of new coronavirus variants.
\newblock \emph{Environmental research}, 206:\penalty0 112240, 2022.

\end{thebibliography}
\end{document}